\documentclass{amsart}
\usepackage{amsmath}
\usepackage{amsfonts}
\usepackage{amsthm}
\usepackage{setspace}
\usepackage{gensymb}
  \usepackage{paralist}
  \usepackage{amssymb} 
  \usepackage{graphics} 
  \usepackage{epsfig} 
 \usepackage[colorlinks=true]{hyperref}
\usepackage{mathrsfs}
\usepackage{empheq} 
\usepackage{graphicx}

\DeclareMathOperator{\Lim}{Lim}

\newcommand{\K}{\mathcal{K}}

\usepackage{ulem}
\newcommand{\comments}[1]{}

\newcommand{\W}{\mathcal{W}}

\newcommand{\A}{\mathcal{A}}

\newcommand{\R}{\mathbb{R}}

\newcommand{\Z}{\mathbb{Z}}

\newcommand{\Cp}{\mathbb{C}}

\newcommand{\w}{\omega}
\newcommand{\Om}{\Omega}

\newcommand{\al}{\alpha}

\newcommand{\be}{\beta}
\newcommand{\ph}{\varphi}
\newcommand{\De}{\Delta}
\newcommand{\de}{\delta}
\newcommand{\s}{\sigma}

\newcommand{\lam}{\lambda}
\newcommand{\gam}{\gamma}
\newcommand{\Gam}{\Gamma}
\newcommand{\kap}{\kappa}

\newcommand{\LP}{\mathcal{P}}

\newcommand{\veps}{\varepsilon}

\newcommand{\smod}{\setminus}

\newcommand{\sub}{\subset}

\newcommand{\goesto}{\rightarrow}

\newcommand{\lv}{\lvert}
\newcommand{\rv}{\rvert}
\newcommand{\del}{\nabla}

\newcommand{\imb}{\hookrightarrow}

\newcommand{\lt}{\lam(t)}
\newcommand{\lams}{\lam(s)}

\newcommand{\conj}[1]{\overline{#1}}
\newcommand{\Sob}[2]{\lVert#1\rVert_{#2}}
\newcommand{\Gev}[2]{\lVert#1\rVert_{\lam,#2}}
\newcommand{\sGev}[3]{\lVert#1\rVert_{\sqrt{\nu#2},#3}}
\newcommand{\gGev}[2]{\lVert#1\rVert_{#2}}

\newcommand{\avg}[1]{\langle#1\rangle}

\newcommand{\noX}[1]{\lVert#1\rVert_{X}}
\newcommand{\noY}[1]{\lVert#1\rVert_{Y}}
\newcommand{\noZ}[1]{\lVert#1\rVert_{Z}}

\newcommand{\abs}[1]{\lvert#1\rvert}

\newcommand{\ls}{\lesssim}
\newcommand{\gs}{\gtrsim}

\newcommand{\lb}{\langle}
\newcommand{\rb}{\rangle}

\newcommand{\ve}[1]{\vec{#1}}

\newcommand{\vz}{\vec{0}}

\newcommand{\lr}{\lambda_a}
\def\ko{\kappa_0}

\newcommand{\req}[1]{(\ref{#1})}

\numberwithin{equation}{section}

\newtheoremstyle{newprop}{\topsep}{\topsep}%
     {}
     {}
     {\bfseries}
     {}
     {3pt}
     {\thmname{Proposition}}
     
   \theoremstyle{definition}

   \newtheorem{prop}{Proposition}

\newtheoremstyle{newdefn}{\topsep}{\topsep}%
     {roman}
     {}
     {\bfseries}
     {}
     {3pt}
     {\thmname{Definition}}
     
   \theoremstyle{definition}

   \newtheorem{defn}{Definition}

\newtheoremstyle{newrmk}{\topsep}{\topsep}%
     {roman}
     {}
     {\bfseries}
     {}
     {3pt}
     {\thmname{Remark}}
     
   \theoremstyle{definition}

   \newtheorem{rmk}[prop]{Remark}

\newtheoremstyle{newlem}{\topsep}{\topsep}%
     {roman}
     {}
     {\bfseries}
     {}
     {3pt}
     {\thmname{Lemma}}
     
   \theoremstyle{definition}

   \newtheorem{lem}[prop]{Lemma}

\newtheoremstyle{newcor}{\topsep}{\topsep}%
     {roman}
     {}
     {\bfseries}
     {}
     {3pt}
     {\thmname{Corollary}}
     
   \theoremstyle{definition}

   \newtheorem{cor}[prop]{Corollary}

\newtheoremstyle{newthm}{\topsep}{\topsep}%
     {roman}
     {}
     {\bfseries}
     {}
     {3pt}
     {\thmname{Theorem}}
     
   \theoremstyle{definition}

   \newtheorem{thm}[prop]{Theorem}

\title[dissipative length scale estimates]
{Dissipative length scale estimates for turbulent flows - a Wiener algebra approach}
\author{A. Biswas$^{1}$}
\address{$^1$Department of Mathematics and Statistics\\
University of Maryland, Baltimore County\\ Baltimore, MD 21250.}

\author{M. S. Jolly$^{2}$}
\address{$^2$Department of Mathematics\\
Indiana University\\ Bloomington, IN 47405}

\author{V. Martinez$^{3}$}
\address{$^2$Department of Mathematics\\
Indiana University\\ Bloomington, IN 47405}

\author{E. S. Titi$^{4}$}
\address{$^4$ Department of Mathematics and Department
of Mechanical and Aerospace Engineering \\
University of California \\
Irvine, California 92697\\
Also:\\
Department of Computer Science and Applied Mathematics\\
Weizmann Institute of Science\\
Rehovot, 76100, Israel}
\address{$\dagger$ corresponding author}

\email[A. Biswas]{abiswas@umbc.edu}
\email[M. S. Jolly]{msjolly@indiana.edu}
\email[V. Martinez]{vinmarti@indiana.edu}
\email[E. S. Titi] {etiti@math.uci.edu}

\begin{document}

\date{\today}


\begin{abstract}
In this paper, a lower bound estimate on the uniform radius of spatial analyticity is established for solutions to the incompressible, forced Navier-Stokes system on an $n$-torus.  This estimate improves or matches previously known estimates provided that certain bounds on the initial data are satisfied.  It is argued that for 2D or 3D turbulent flows, the initial data is guaranteed to satisfy these hypothesized bounds on a significant portion of the 2D global attractor or  the 3D weak attractor.  In these scenarios, the estimate obtained for 3D generalizes and improves upon that of \cite{dt}, while in 2D, the estimate matches the best known one found in \cite{kuka}.  A key feature in the approach taken here, is the choice of the Wiener algebra as the phase space, i.e. the Banach algebra of functions with absolutely convergent Fourier series, whose structure is suitable for the use of the so-called Gevrey norms.
\end{abstract}

\maketitle

\section{Introduction}
The conventional theory of turbulence posits the existence of certain universal length scales of paramount importance.  For instance, according to Kolmogorov, there exists
a  {\it dissipation length scale}, $\lambda_d$, beyond which 
the viscous effects dominate the nonlinear coupling.  This length scale  can be characterized by the exponential decay of  the energy density.  Consequently, one expects the dissipation wave-number, $\kap_d=\lam_d^{-1}$, to majorize the inertial range where energy consumption is largely governed by the nonlinear effects and dissipation can be ignored.

In \cite{foias, dt} it is shown 
that as characterized by Gevrey norms, the (uniform) radius of spatial analyticity, here denoted $\lr$, provides a lower bound for the dissipation length scale, i.e., $\lr \le \lambda_d$.  The space analyticity radius has been well-studied over the years, especially after the pioneering work of Foias and Temam in \cite{ft}, where they presented a novel Gevrey norm approach to establish analyticity of solutions to NSE in both space and time.  An advantage of this approach is that it avoids having to make cumbersome recursive estimates on derivatives.  Consequently, it has become a standard tool in estimating the analyticity radius for various equations (cf. \cite{fet, ot, ot2, lo, bs2, bis2, kv1, lt}).

Kolmogorov's theory for 3D turbulence asserts that 
\begin{align}\label{grail3D}
\lambda_d\sim \lambda_\veps:=(\nu^3/\veps)^{1/4}\;,
\end{align}
where $\nu$ is viscosity and $\veps$ is the mean {\it energy dissipation rate} per unit mass. 

For 3D decaying turbulence, it is shown in
\cite{dt} that 
\begin{align}\label{grailDT}
\lr\sim \ko^{-1}(\ko\tilde{\lambda}_\veps)^4\;,
\end{align}
where $\tilde{\lambda}_\veps$ is as in \eqref{grail3D}, except that the energy dissipation rate is a supremum in time rather than an averaged quantity (see \eqref{vepsdef}, \eqref{vepssup}).  The more significant discrepancy is a power of 4 versus a power of 1 in \eqref{grail3D}. Our improvement is done under the 2/3-power law assumption \eqref{powerlaw3D} on the energy spectrum for a forced, turbulent flow, by means of an ensemble average with respect to an invariant measure.
It is valid on a portion of the attractor (weak in the 3D case);
the significance of which is quantified in terms of this measure.  Ultimately, we 
conclude that 
\begin{align}\label{near3D}
\lr \gtrsim_p \ko^{-1}(\ko \lambda_\veps)^{59/24}
\end{align}
holds with probability  $1-p$,
where the suppressed constant in the inequality tends to $0$ as $p \to 0$.
Similarly, a heuristic scaling argument by Kraichnan  for
2D turbulence leads to 
\begin{align}\label{grail2D}
\lambda_d\sim \lambda_\eta:=(\nu^3/\eta)^{1/6}\;,
\end{align}
where $\eta$ is the mean {\it enstrophy dissipation rate} per unit mass.
We show that if the 2D power law \eqref{power} for the energy spectrum holds, then
\begin{align}\label{near2D}
\lr \gtrsim_p \ko^{-1}(\ko \lambda_\eta)^{2}
\end{align}
up to a logarithm in $G$.

 These improved estimates actually follow from more general bounds on the radius of
analyticity which require the solution to satisfy a certain ``smallness" condition.  Those 
conditions are met under the power law assumptions, when averaged with respect to an invariant measure.   Kukavica \cite{kuka} achieved the same bound in 2D up to a 
logarithmic correction on all of the attractor using complex analytic techniques, interpolating between $L^p$ norms of the initial data and the complexified solution, and invoking the theory of singular integrals.

The approach in \cite{kuka} was actually a modification of the approach in \cite{gk}, where it was shown that $\lam_d\gtrsim \nu(\sup_{t\leq T^*/2}\Sob{u(t)}{L^\infty})^{-1}$.
 It is interesting to ask if these estimates can be obtained by working exclusively in frequency-space using Fourier techniques, rather than in physical space with the $L^\infty$ norm.  Indeed, this is an impetus of our work.

The technique applied here combines the use of Gevrey norms with the semigroup approach of Weissler \cite{weissler1}.  Motivated by recent developments, we work over a subspace of the Wiener algebra, whose norm is a Sobolev-Gevrey-type norm in $\ell^1$ (see \req{gev}).  This norm and approach was applied in \cite{bis} to study spatial analyticity and Gevrey regularity of solutions to the NSE.  However, the resulting estimate on the spatial radius of analyticity was not optimal for large data.  This approach is refined here to obtain a sharper estimate for such data.  The advantage of working in the Wiener algebra, $\W$, i.e. the Banach algebra of functions whose Fourier series converge absolutely, was explored in \cite{ot}, where a sharp estimate on the radius of analyticity was obtained, for instance, for real steady states of the nonlinear Schr\"odinger equations.  More recently, these $\ell^1$-based Gevrey norms were also applied to the Szeg\"o equation in \cite{guotiti1} and the quasi-linear wave equation in \cite{guotiti2}.  In \cite{guotiti1}, an essentially sharp estimate on the radius is obtained there as well.  While these works used energy-like approaches, the effectiveness and robustness of $\W$ as a working space to study analyticity has become increasingly clear.

There are several advantages to our approach.  First, our method is quite elementary.  Since $\W$ is embedded in $L^\infty$, we essentially recover the results of \cite{gk} and \cite{kuka} without resorting to complex-analytic techniques and the theory of singular integrals, while furthermore allowing for rougher initial data.  Secondly, this approach also applies to the case $1<p<\infty$, thereby unifying the results of \cite{dt}, \cite{ft}, \cite{gk}, and \cite{kuka} .  Thirdly, no logarithmic corrections appear in our estimates initially; they only appear when specializing to the context of 3D or 2D turbulence (see \req{just:2}).  Finally, the method is rather robust and applies to a wide class of active and passive scalar equations with dissipation, including the quasigeostrophic (QG) equations.  Note that in the case of QG with supercritical dissipation, the method will only accommodate subanalytic Gevrey regularity (see \cite{vthesis}).

\section{Preliminaries}\label{prelim}
The Navier-Stokes system in $\Om:=[0,L]^n$ for $n>1$ is given by
	\begin{align}\label{nse:sys}
		\begin{cases}	
		u_t-\nu\De u+u\cdotp\del u+\del p=F\\\
			\del\cdotp u=0\\
			u(x,0)=u_0(x)
		\end{cases}
	\end{align}
where $u_0:\Om\goesto\Om$ and $f:\Om\times[0,T)\goesto\Om$ are given, and $p:\Om\times[0,T)\goesto[0,L]$ and $u:\Om\times[0,T)\goesto\Om$ are unknown. We assume that $u_0,u,p,F$ are all $L$-periodic and mean-zero, and that $u_0$ is divergence-free.  

We will use the so-called {wave-vector form} of \req{nse:sys}, which is simply \req{nse:sys} written in terms of its Fourier coefficients
	\begin{align}\label{wv:sys}
		\begin{cases}
		\frac{d}{dt}\hat{u}(k,t)=-\nu\kap_0^2|k|^2\hat{u}(k,t)+B[{\ve{u}},{\ve{u}}](k,t)+\widehat{{f}}(k,t)\\
		k\cdotp\hat{u}(k,t)=0\\
		\hat{u}(k,0)=\hat{u}_0(k),
		\end{cases}
	\end{align}
where $k\in\Z^n$, $\ve{u}:\Z^n\times[0,T)\goesto\Cp^n$ such that $\ve{u}(t)=(\hat{u}(k,t))_{k\in\Z^n}$, and $f=\LP F$, where $\LP$ is the Helmholtz-Leray orthogonal projection, i.e. projection onto divergence-free vector fields,
	\begin{align}\label{leray}
		\LP(\hat{u}(k)e^{i\kap_0k\cdotp x})=\left(\hat{u}(k)-\left(\frac{k}{|k|}\cdotp\hat{u}(k)\right)\frac{k}{|k|}\right)e^{i\kap_0k\cdotp x},\quad  (k \in \Z^n).
	\end{align}
Recall also that the mean zero condition forces $\hat{u}(0,t)=0$ for all $t$.
The bilinear term $B$ has Fourier coefficients given by
	\begin{align}
			B[\ve{u},\ve{v}](k,t)e^{i\kap_0k\cdotp x}:=i\kap_0\LP\left(\sum_{\ell\in\Z^n\smod\{\vz\}}(k\cdotp\hat{u}(\ell,t))\hat{v}(k-\ell,t)e^{i\kap_0k\cdotp x}\right),
	\end{align}
Note that $\ve{B}[\ve{u},\ve{v}]$ will denote the sequence $(B[\ve{u},\ve{v}](k))_{k\in\Z^n}$.

Observe that
	\begin{align}\label{leray:ineq}
		\lvert \widehat{\LP u}(k)\rvert\ls\abs{\hat{u}(k)},
	\end{align}
and also that the following basic convolution estimate holds
	\begin{align}\label{nonlin:ineq}
		\lv B[\ve{u},\ve{v}](k)\rv\ls \kap_0\lv k\rv(\lv \ve{u}\rv*\lv \ve{v}\rv)(k)\ \text{for all}\ k\in\Z^n.
	\end{align}

Since we will be working with \req{wv:sys}, we choose an appropriate sequence space as our ambient space.  Define
	\begin{align}\label{K}
		\K&:=\{(\hat{u}(k))_{k\in\Z^n}\in(\Cp^n)^{\Z^n}: \hat{u}(0)=0,\hat{u}(k)=\hat{u}(-k)^*,k\cdotp\hat{u}(k)={0}\},
	\end{align}
where $\hat{u}(k)^*=(\conj{\hat{u}_1(k)},\dots,\conj{\hat{{u}}_n(k)})$.  For  $\s\in\R$ define
	\begin{align}\label{sob:sp}
		V_{\s}&:=\{(\hat{u}(k))_{k\in\Z^n}\in(\Cp^n)^{\Z^n}:\Sob{\ve{u}}{\s}<\infty\}\cap\K,
	\end{align}
where
	\begin{align}\label{sob:norm}
		\Sob{\ve{u}}{\s}&:=\kap_0^\s\sum_{k\in\Z^n}\lv k\rv^{\s}\lv\hat{u}(k)\rv.
	\end{align}
and $\ve{u}$ denotes an element of $(\Cp^n)^{\Z^n}$.  Observe that when $\s=0$, the norm on $V_\s$ agrees with that on the Wiener algebra, i.e.
	\begin{align}
		(\nu\kap_0)^{-1}\Sob{\ve{u}}{0}=\Sob{{u}}{\W},
	\end{align}
where $u$ is the continuous function whose Fourier coefficients are given by $\hat{u}(k)$.  In fact, we have $V_\s\sub\W\cap\K\sub V_{-\s}$, for all $\s\geq0$.

For $\ve{u}\in V_\s$, we define the \textit{(analytic) Gevrey norm} of $\ve{u}$ by
	\begin{align}\label{gev}
		\Sob{\ve{u}}{\lam,\s}&:=\kap_0^\s\sum_{k\in \Z^n}e^{\lam\kap_0\lv k\rv}\lv k\rv^{\s}\lv\hat{u}(k)\rv
	\end{align}
for $\lam\geq0$.  Observe that $\lam$ has the physical dimension of $\textit{length}$.

For a time-dependent sequence $\ve{u}(\ \cdotp)$ such that $\ve{u}(t)\in V_\s$, for all $t\geq0$, we define the \textit{(analytic) Gevrey norm} of $\ve{u}(t)$ by
	\begin{align}\label{t:gev}
		\Sob{\ve{u}(t)}{\lam(t),\s}&:=\kap_0^\s\sum_{k\in \Z^n}e^{\lam(t)\kap_0\lv k\rv}\lv k\rv^{\s}\lv\hat{u}(k,t)\rv,
	\end{align}
where  $\lam:\R^+\goesto\R^+$ is increasing and sublinear, i.e.  $\lam(s+t)\leq\lam(s)+\lam(t)$ for all $s,t\geq0$.  Observe that
	\begin{align}
		\Sob{u(t)}{\W}\ls_\s \lam(t)^{\ph(\s)}\frac{\kap_0^{-\s}}{\nu\kap_0}\Sob{\ve{u}(t)}{\lam(t),\s},
	\end{align}
for all $\s\in\R$ and $t>0$, where $\ph(\s)=\s$ if $\s<0$ and $0$ otherwise.

 It is well-known that the Gevrey norm characterizes analyticity, a fact stated 
 more precisely in the following proposition (cf. \cite{lo}, \cite{katznel}):

\begin{prop}\label{analyt}
Let $\s\in\R$.
	\begin{enumerate}
		\item If $\Sob{{\ve{u}}}{\lam,\s}<\infty$, then  $u$ admits an analytic extension on $\{x+iy:|y|<\lam\}$;
		\item If $u$ has an analytic extension on $\{x+iy:|y|<\lam\}$, then $\Sob{\ve{u}}{\lam',\s}<\infty$ for all $\lam'<\lam$.
	\end{enumerate}
\end{prop}

In particular, if a function has finite Gevrey norm, then the Fourier modes decay exponentially.  Indeed, if $\Sob{\ve{u}}{\lam,\s}<\infty$, then
	\begin{align}\label{exp:decay}
		|\hat{u}(k)|\leq e^{-\lam|k|}|k|^{-\s}\Sob{\ve{u}}{\lam,\s}.
	\end{align}

\begin{defn}
If $u$ is analytic, then we define
	\begin{align}
		\lam_{\text{max}}=\sup\{\lam'>0:\Sob{\ve{u}}{\lam',\s}<\infty\}
	\end{align}
  to be the the maximal (uniform) radius of spatial analyticity of $u$.
Moreover, due to \req{exp:decay} we have $\lambda_d \geq \lambda_{\text{max}}$.

\end{defn}
\begin{rmk}
For convenience, we adopt the following conventions for the rest of the paper.  
	\begin{enumerate}
		\item We will usually write $\ve{u}$ simply as $u$, which is the function whose Fourier series have modes $\hat{u}(k)$, for $k\in\Z^n$.
		\item By  $u(t)$ or $u(k)$, or when the context is clear, simply $u$, we shall mean the time-dependent sequence $\ve{u}(t)=(\hat{u}(k,t))_{k\in\kap_0\Z^n}$, unless otherwise specified.  
		\item We will use $\ls$ to suppress extraneous absolute constants or physical parameters.  In some instances, the dependence of these constants will be indicated as subscripts on $\ls$.
		\item We will also use the notation $\sim$ to denote that the two-sided relation $\ls$ and $\gs$ holds.
	\end{enumerate}
\end{rmk}

For $1\leq q\leq \infty$ and $0<T_f\leq\infty$, we define
	\begin{align}
		M_{0}&:=\frac{\kap_0^{-\s}}{\nu\kap_0}\Sob{u_0}{\s},\label{M0}\\
		M_{f}&:=\begin{cases}\frac{\kap_0^{-\s}}{\nu^2\kap_0^3}{\left(\nu\kap_0^2\int_0^{T_f}\gGev{f(s)}{\lams,\s}^q\ {ds}\right)^{1/q}},	&1\leq q<\infty\\
		\frac{\kap_0^{-\s}}{\nu^2\kap_0^3}{\sup_{0\leq t\leq T_f}\gGev{f(t)}{\lt,\s}}
		&q=\infty
				\end{cases}\label{M:data}
	\end{align}
and
	\begin{align}\label{nd:M}
		M&:=M_0+M_{f}.
	\end{align}

For any dimension $n>0$, the Grashof number is defined as
	\begin{align}\label{gras}
		 G:=\frac{\kap_0^{n/2}}{\nu^2\kap_0^3}\sup_{0\leq t\leq T_f}\lVert f(t)\rVert_{L^2}.
	\end{align}
Observe that $M$ and $G$ are dimensionless.  One can show that when $f$ is time-independent and has only finitely many modes, i.e. $f=P_{\bar{\kap}}f$, where 
	\begin{align}\label{proj}
		P_{\bar{\kap}} f:=\sum_{|k|\leq\bar{\kap}/\kap_0}\hat{f}(k)e^{i\kap_0k\cdotp x},
	\end{align}
then $M_f$ is comparable to $G$ up to a constant depending on only $\kap_0, \bar{\kap}$, a fixed parameter $\tau$, and $\lam_f$, where $\lam_f$ satisfies
	\begin{align}\label{lamf}
		\sup_{|y|\leq\lam_f}\Sob{f(\ \cdotp+iy)}{L^2}<\infty; 
	\end{align} 
see Proposition \ref{M:gras} in Appendix.

Now suppose that data $u_0$ and $f$ are given such that $M<\infty$.  Let $A$ be the Stokes operator, $A:=-\mathcal{P}\De$, where $\mathcal{P}$ is defined as in \req{leray}.  Then the heat kernel, $e^{\nu tA}$, is the Fourier multiplier defined by
	\begin{align}\label{heat:kernel}
		\widehat{e^{\nu tA}u}(k):=e^{-\nu t\kap_0^2\abs{k}^2}\hat{u}(k),
	\end{align}
or equivalently, $e^{\nu tA}\ve{u}=(e^{-\nu t\kap_0^2\abs{k}^2}\hat{u}(k))_{k\in\Z^n}$.  We will use two notions of solutions to \req{wv:sys}.

\begin{defn}
For $0<T\leq\infty$, a \textit{mild solution} to \req{wv:sys} is any function $\ve{u}\in C([0,T];\K)$ such that
	\begin{align}
		\int_0^te^{-\nu(t-s)\kap_0^2\abs{k}^2}|B[\ve{u},\ve{u}](k,s)|\ ds<\infty,
	\end{align}	
 for all $k\in\Z^n$, and
	\begin{align}\label{mild}
		\ve{u}(t)=e^{-\nu tA}\ve{u}_0+\int_0^te^{-\nu(t-s)A}\ve{\LP{f}}(s)\ ds-\int_0^te^{-\nu(t-s)A}\ve{B}[\ve{u},\ve{u}](s)\ ds,
	\end{align}
for all $0\leq t\leq T$.
\end{defn}

\begin{defn}\label{wk:sol}
For $0<T\leq\infty$, a \textit{weak solution} to \req{wv:sys} is any function $\ve{u}\in C([0,T];\K)$ such that
	\begin{align}
		B[\ve{u},\ve{u}](k,t)<\infty
	\end{align}	
for all $k\in\Z^n$ and a.e. $t\in[0,T]$ and 
	\begin{align}\label{weak}
		\frac{d}{dt}\hat{u}(k,t)=-\nu\kap_0^2\lv k\rv^2\hat{u}(k,t)-B[\hat{u},\hat{u}](k,t)+\hat{{f}}(k,t)
	\end{align}
holds for all $k\in\Z^n$ and a.e. $t\in[0,T]$.
\end{defn}

The fact that Definition \ref{wk:sol} is equivalent to the usual definition of weak solution for a periodic flow can be found in \cite{temamnse}. 

Finally, we define the regularity that we ultimately seek to establish.

\begin{defn}\label{gev:reg}
A mild  or weak solution $\ve{u}\in C([0,T];\K)$ of \req{wv:sys} is \textit{Gevrey regular} if there exists $\s\in\R$ and sublinear $\lam:\R^+\goesto\R^+$ such that
	\begin{align}
		\sup_{0\leq t\leq T}\Sob{\ve{u}(t)}{\lam(t),\s}<\infty.
	\end{align}
\end{defn}

\section{Main Theorems}\label{sect:main}

We first state a result for a general force.  

\begin{thm}\label{gen:thm2}
Let $1<q\leq\infty$ and $-1<\s\leq0$ and $M,T_f$ be as defined in \req{nd:M}.  Suppose that $u_0$ and $f$ are given such that $M<\infty$.  Then for some $0<T^*\leq T_f$, there exists a mild solution $u\in C([0,T^*];V_\s)$ to \req{wv:sys}, which is also a Gevrey regular weak solution, with radius of analyticity at time $T^*$ satisfying
	\begin{align}\label{M:rad1}
		\lr\gs \kap_0^{-1}\cdotp\begin{cases}M^{-1/(1-2|\s|/q')},&1<q\leq2,\\
						M^{-1/(1-|\s|)},&2\leq q\leq\infty,
						\end{cases}
	\end{align}
where $1/q':=1-1/q$.  Moreover, there exists a constant $C^*$ such that if $M\leq C^*$, then one may take $T^*=T_f$.  In this case, the solution exists for all $0\leq t\leq T_f$ and the radius of analyticity at time $t$ satisfies
	\begin{align}\label{rad:infty}
		 \lr\gs\sqrt{\nu t}.
	\end{align}
\end{thm}

In the case where the forcing is time-independent and has finitely many modes, we can express the estimate on the radius of analyticity in terms of the Grashof number, provided
a ``smallness" condition on the solution holds.

\begin{thm}\label{gras:peq1}
Suppose that  $f$ is time-independent and satisfies $f=P_{\bar{\kap}}f$.  If
	\begin{align}
		\Sob{u_0}{\W}\ls G^{1/2},
	\end{align}
then for some $0<T^*<(\nu\kap_0^2)^{-1}$, there exists a unique weak solution $u\in C([0,T^*],V_0)$ to \req{nse:sys} such that $u$ is Gevrey regular and the radius of analyticity at time $T^*$ satisfies
	\begin{align}\label{gras:rad:2d}
		\lr\gs_{\bar{\kap},\kap_0}\kap_0^{-1}G^{-1/2}.
	\end{align}
\end{thm}

The following estimate is not as sharp, but holds under a weaker ``smallness" condition.

\begin{thm}\label{gras:peq1:3d}
Suppose that  $f$ is time-independent and satisfies $f=P_{\bar{\kap}}f$.  If
	\begin{align}
		\Sob{A^{-3/8}u_0}{\W}\ls \kap_0^{-3/4}G^{11/16},
	\end{align}
where $A=-\Delta$ with periodic boundary conditions, is the Stokes operator,
then for some $0<T^*<(\nu\kap_0^2)^{-1}$, there exists a weak solution $u\in C([0,T^*],V_{-3/4})$ to \req{nse:sys} such that $u$ is Gevrey regular and the radius of analyticity at time $T^*$ satisfies
	\begin{align}\label{gras:rad:3d}
		\lr\gs_{\bar{\kap},\kap_0}\kap_0^{-1}G^{-59/64}.
	\end{align}
\end{thm}

\begin{rmk}\label{rmk:thms}
One can also have $\s>0$ in Theorem \ref{gen:thm2} (see its proof in Section \ref{proofs}).  In fact, a more general version of Theorem \ref{gras:peq1} and \ref{gras:peq1:3d} is proved in Section \ref{proofs} (see Theorem \ref{gen:thm3}).  

The estimate on $\lr$ in Theorem \ref{gen:thm2} can be compared to the one in \cite{bis} when $q=2$.  However, in that work their choice of $\lam(t)$ (as in the Definition \ref{gev:reg}) yielded instead the estimate
	\begin{align}
		\lr\gs\kap_0^{-1}M^{-2/(1-|\s|)},
	\end{align}
which is less sharp than the corresponding estimate in \req{M:rad1} when $M$ is large.

One should also note that if $C^*$ is too small, then the global attractor in 2D becomes trivial (cf \cite{dfj1, marchioro}).  Physically, this corresponds to the case of decaying turbulence.  Nevertheless, if $M$ is sufficiently small, then $T_f=\infty$ is allowed, in which case the solution exists globally in time with radius that grows without bound in time as $\sqrt{\nu t}$.

Uniqueness of weak solutions to \req{nse:sys} is guaranteed in two-dimensions, but in 3D
is still an open question.  There are, however, cases where the uniqueness is guaranteed in any dimension (see \cite{temamnse} pp. 298-99).  In particular, as long as $\s\geq0$, the solution of Theorem \ref{gen:thm2} is unique in the class of weak solutions.

In the case where the force is identically zero, one can employ energy techniques as in \cite{dt}, \cite{ft} and obtain 
	\begin{align}\label{nrg:met}
		\lr\geq C\frac{\kap_0^{-1}}{\Sob{u_0}{\W}}
	\end{align} 
where $\lr$ represents the radius of analyticity at some time $T^* $ strictly less than the maximal time of existence.  The constant here can be explicitly identified as $C=\log(1+\gam)/\sqrt{\gam}$, where $\gam$ is the nontrivial solution to 
	\[
		(2\gam)^{-1}\log(1+\gam)-(1+\gam)^{-1}=0.
	\]
Note that \req{nrg:met} is precisley the estimate in \req{M:rad1} (up to an absolute constant).  The energy approach, however, encounters technical difficulties when one includes forcing on infinitely many scales.  The reader is referred to \cite{vthesis} for additional details.

In \cite{vthesis}, the estimates are also done in $\ell^p$ for $1<p<\infty$.  In particular, when $n=3, p=2, \s=1$, the result of \cite{dt} is generalized to include forcing on all scales, and the estimate on the radius is the same as the one derived there (up to an absolute constant).  One can make an argument similar to the one presented in Section \ref{3d:turb} that would justify the corresponding assumption on the initial data, but working on the 3D weak attractor.  For background on the weak attractor, see \cite{dfj2} or \cite{fmrt:book}.

Finally, the techniques used to prove Theorem \ref{gras:peq1} 
apply equally well to the vorticity formulation of Navier-Stokes, the case of fractional dissipation, and a wide class of active and passive scalar equations, including 2D dissipative QG equations, (see \cite{vthesis}).  These techniques also apply to the case $\Om=\R^n$ (see \cite{bis2}).  For more results on the subcritical QG, see for instance \cite{ccw}, where analyticity is established for arbitrary initial data in $H^2$, or \cite{dongdong}, where a local smoothing effect is exploited to establish analyticity, or \cite{bis2}, where analytic Gevrey regularity is established for several other equations as well.  For results on the analyticity of solutions for critical QG equations, see \cite{hdong} and \cite{kiselev:note}.  For results on the regularity of passive scalar equations see \cite{silv1} or \cite{silv:vicol:zlat}.  The classical Hilbert space techniques of \cite{ft} have also been successfully applied to the Euler equations (see \cite{kv1} and \cite{lo}).
\end{rmk}

\subsection{Application  to Turbulent Flows}\label{3d:turb}
In this subsection, we show how our results in Theorems  \ref{gras:peq1}, \ref{gras:peq1:3d} improves the known estimates for $\lambda_d$ for turbulent flows.
While their ``smallness" assumptions
may not hold on all of the 2D global (3D weak) attractor, in the context of turbulence, one can expect these conditions to hold \textit{on average}, in a precise sense. 

The statistical theory of turbulence concerns relations between 
quantities that are averaged, either with respect to time or over an ensemble of flows, e.g. results from repeated experiments. It is remarkable that these two seemingly different approaches are in fact related.

 The mathematical equivalent of a large time average 
is rigorously expressed in terms of Banach limits.  Following \cite{fmrt:book}, define the space $H$ by
	\begin{align}
		H:=\{(\hat{u}(k))_{k\in\Z^n}\in(\Cp^n)^{\Z^n}:\Sob{\ve{u}}{\ell^2}<\infty\}\cap\K.
	\end{align}
Let $\Phi$ be a real-valued weakly continuous function on $H$.  Then for any weak solution $u$ of \req{wv:sys} on $[0,\infty)$, there exists a probability measure $\mu$ for which
	\begin{align}\label{time:avg}
		\lb\Phi\rb:=\int_H\Phi(u)\ d\mu(u)={\Lim_{T\goesto\infty}}\frac{1}T\int_0^T\Phi(u(t))\ dt,
	\end{align} 
where $\Lim$ is a Hahn-Banach extension of the classical limit.  The measure $\mu$ is called a \textit{time-average measure} of $u$.  Note that neither $\Lim$ nor $\mu$ are unique.  The use of $\Lim$ surmounts the technical difficulty that the limit in the usual sense may not exist.  If $u$ is weak solution to the 2D NSE, then by regularity of such solutions, one can work in the strong topology on $H$.  Moreover, by uniqueness, one can show that $\mu$ is in fact \textit{invariant} with respect to the corresponding semigroup, i.e. $\mu(E)=\mu(S(t)^{-1}E)$ for all $t\geq0$, for all measurable sets $E\sub H$.  Thus, a time average measure is also a so-called {\it stationary statistical solution} of the NSE.  For a more detailed background see \cite{fmrt:book}.

We now specialize to the cases of 3D and 2D turbulence, and interpret the main theorems in those settings.

\subsubsection{3D Turbulence}
The mean energy dissipation rate per unit mass is defined as
	\begin{align}\label{vepsdef}
		\veps:=\nu\kap_0^3\lb\Sob{\del u}{L^2}^2\rb\;.
	\end{align}
In 3D, Kolomogorov argued that because one can ignore nonlinear effects in the dissipation range, the length scale indicating where dissipation is the dominant effect should depend solely on $\veps$ and $\nu$.  By a simple dimensional argument, one then arrives at 
	\begin{align}\label{kol}
		\lam_\veps=\left(\frac{\nu^3}{\veps}\right)^{1/4}.
	\end{align}
 In other words, according to Kolmogorov, for turbulent flows in 3D, $\lambda_d \sim \lambda_{\veps}$ with $\lambda_{\veps}$  given in \eqref{kol}. We will now describe the best known rigorous result in this direction.

In \cite{dt}, the radius of analyticity was estimated in terms of $\veps_{\text{sup}}$ as
	\begin{align}\label{dt:est}
		\lr\gtrsim\frac{(\nu\kap_0)^3}{\veps_{\text{sup}}}.
	\end{align}
where
	\begin{align}\label{vepssup}
		\veps_{\text{sup}}:=\nu\kap_0^3\sup_{0\leq t\leq T^*/2}\Sob{\nabla u(t)}{L^2}^2\;.
	\end{align}
represents the largest instantaneous energy dissipation rate (per unit mass) up to time $T^*/2$, and $T^*$ is the maximal time of existence of a regular solution.  A heuristic
argument is given to support $\veps_{\text{sup}}\sim\veps$ as in \cite{dt}, then \req{dt:est} becomes
	\begin{align}
		\lam_d\gtrsim\kap_0^{-1}(\kap_0\tilde{\lam}_\veps)^4\;, \quad \text{where} \quad
		\tilde{\lambda}_\veps=\left(\frac{\nu^3}{\veps_{\text{sup}}}\right)^{1/4}
	\end{align}
It is not presently known if $\veps_{\text{sup}}$ remains finite beyond $T^*$.  Hence, it is not possible to obtain an estimate of the smallest length scale for an arbitrary weak solution.  In fact, it is not possible to extend these estimates on the weak attractor either since it is not known whether or not a trajectory, i.e. a weak solution defined for all $t\in\R$, is regular.  However, it is well-accepted that statements regarding length scales in turbulence actually concern ``averages" and not specific trajectories (cf. \cite{foias:prodi,fjmr, fjmrt, bfjr}, or \cite{fmrt:book, frisch} for introductory approaches).  Indeed, this is the thrust of our current discussion.

 In addition to the dissipation range and wave number, another basic tenet in the Kolmogorov theory of turbulence is the so-called power law for the energy spectrum. More specifically, let
$\bar{\kappa}$ denote the wave number in which energy is injected into the flow, i.e., $f=P_{\bar{\kappa}}f$. Denote the Kolmogorov wave-number  $\kap_{\veps}:=1/\lam_{\veps}$.
 Then the range of wave-numbers $[\bar{\kap},\kap_{\veps}]$ is known as the inertial range in which the effect of viscosity is negligible. The nonlinear (inertial) term simply transfers the energy injected into the flow through the inertial range at a rate of $\veps$. Moreover, defining the quantity
\begin{align*}
		\ \ \text{e}_{{\kap_1},\kap_{2}}:=\kap_0^3\lb\Sob{(P_{\kap_2}-P_{{\kap_1}})u}{L^2}^2\rb,
	\end{align*}
	the well-celebrated Kolmogorov's power law asserts that a turbulent flow must satisfy the relation
	\begin{align}\label{powerlaw3D}
		\text{e}_{\kap,2\kap}&\sim\veps^{2/3}/\kap^{2/3},\ \text{for} \ \kap\in[\bar{\kap},\kap_{\veps}].
		\end{align}
		Additionally, it is also known that if the Grashof number is sufficiently small, then the flow is not turbulent and the attractor in this case consists of only one point. In view of this discussion, we {\it define} a flow to be turbulent if the Kolmogorov power law holds and the Grashof number is sufficiently large, i.e.
	\begin{align}\label{large:g2}
		G\gtrsim\left(\frac{\bar{\kap}}{\kap_0}\right)^{3/2}.
	\end{align}
	 It is shown in \cite{dfj2} that for such a flow
one necessarily has the bounds
	\begin{align}
		\frac{\nu^2}{\kap_0}\left(\frac{\kap_0}{\bar{\kap}}\right)^{5/2}G\ls\lb&\Sob{u}{L^2}^2\rb\ls\frac{\nu^2}{\kap_0}\left(\frac{\kap_0}{\bar{\kap}}\right)G,\label{turb:est1}\\
		\nu^2\kap_0\left(\frac{\kap_0}{\bar{\kap}}\right)^{11/4}G^{3/2}\ls\lb\Sob{&A^{1/2} u}{L^2}^2\rb\ls\nu^2\kap_0\left(\frac{\kap_0}{\bar{\kap}}\right)^{1/2}G^{3/2}.\label{turb:est2}
	\end{align}

The following is the main result of this section which improves upon the estimate in \cite{dt} for 3D turbulent flows.
\begin{thm}
Let $\mu$ be a time-average measure for a 3D turbulent flow and let $0<p<1$. There exists a set
$S \subset {\mathcal A}_w$ with $\mu(S) \ge 1-p$ such that
\begin{align*}
		\lam_d(u)\gs_p\kap_0^{-1}(\kap_0\lam_\veps)^{59/24}\  \mbox{for all}\ u \in S.
	\end{align*}
\end{thm}

\begin{proof}
Recall that Theorem \ref{gras:peq1:3d} ensures that
	\begin{align}
		\lr\gs\kap_0^{-1}G^{-59/64},
	\end{align}
provided that the initial data satisfies
	\begin{align}\label{neg:sob}
		\Sob{A^{-3/8}u_0}{\W}\ls \kap_0^{-3/4}G^{11/16}.
	\end{align}
We argue that \req{neg:sob} is guaranteed to hold on a significant portion of the 3D weak attractor, $\A_w$.  We now quantify the likelihood that \req{neg:sob} occurs within $\A_w$ with respect to any time-average measure $\mu$.

First, observe that by Proposition \ref{agmon} with $\s=-3/4$, one has the inequality
	\begin{align}\label{interpolate}
		\Sob{A^{-3/8}u}{\W}\ls\Sob{u}{L^2}^{1/4}\Sob{A^{1/2} u}{L^2}^{3/4}.
	\end{align}
Let $0<p<1$ and define the following sets
	\begin{align}
		A_p&:=\left\{u\in\A_w:\Sob{u}{L^2}\gs\sqrt{\frac{2}{p}}\nu\kap_0^{-1/2}\left(\frac{\kap_0}{\bar{\kap}}\right)^{1/2}G^{1/2}\right\}\notag\\
		B_p&:=\left\{u\in\A_w:\Sob{A^{1/2} u}{L^2}\gs\sqrt{\frac{2}p}\nu\kap_0^{1/2}\left(\frac{\kap_0}{\bar{\kap}}\right)^{1/4}G^{3/4}\right\}.\notag
	\end{align}
Then by \req{turb:est1}, \req{turb:est2}, and Chebyshev's inequality
	\begin{align}
		\mu(A_p)\leq \frac{p}2\ \text{and}\ \mu(B_p)\leq \frac{p}2.\notag
	\end{align}
We note that the support of $\mu$ is contained in $\A_w$ (see \cite{fmrt:book}), so that \req{turb:est1} and \req{turb:est2} ensure that these inequalities are not trivial.  It follows that
	\begin{align}
		\mu((\A_w\smod A_p)\cap (\A_\w\smod B_p))\geq1-p.\notag
	\end{align}
This combined with \req{interpolate} implies that
	\begin{align}
		\mu\left\{u\in\A_w:\Sob{A^{-3/8}u_0}{\W}\ls\sqrt{\frac{2}{p}}\nu\kap_0^{1/4}G^{11/16}\right\}\geq1-p.
	\end{align}
Then Theorem \ref{gras:peq1:3d} gives
	\begin{align}
		\mu\left\{u\in\A_w:\lr(u)\gs_{p} \kap_0^{-1}G^{-59/64}\right\}\geq1-p,
	\end{align}
where $\gs$ suppresses a constant which tends to $0$ as $p\goesto0$.  Finally, observe that \req{turb:est2} implies that $\veps\gs\nu^3\kap_0^4G^{3/2}$, so that $(\kap_0\lam_\veps)^{8/3}\ls G^{-1}$.  Therefore
	\begin{align}
		\mu\left\{u\in\A_w:\lr(u)\gs_{p} \kap_0^{-1}(\kap_0\lam_\veps)^{59/24)}\right\}\geq1-p,
	\end{align}
where $\lr(u)$ denotes the radius of analyticity of $u$ at time $T^*=T^*(u)$.

In particular, we have just shown that for any $u_0\in(\A_w\smod A_p)\cap (\A_\w\smod B_p)$, the radius of analyticity for the corresponding solution at time $T^*$ is bounded below by
	\begin{align*}
		\lr\gs_p\kap_0^{-1}(\kap_0\lam_\veps)^{59/24},
	\end{align*}
provided that we are in the turbulent scenario described above.
\end{proof}

\subsubsection{2D Turbulence}\label{2d:turb}

 In the Kraichnan theory of 2D turbulence enstrophy $\|u\|^2$ is also dissipated,
and it does so at a mean rate per unit mass given by
$$
\eta=\nu\ko^2\langle |Au|^2\rangle\;.
$$
Two key wave numbers are 
	\begin{align*}
		\kap_{\eta}:=\left(\frac{\eta}{\nu^2}\right)^{1/6}\sim\left(\frac{\langle \Sob{A u}{L^2}^2\rangle}{L^2\nu^2}\right)^{1/6},\ \ \kap_\s&:=\left(\frac{\langle\Sob{ Au}{L^2}^2\rangle}{\langle\Sob{ A^{1/2}u}{L^2}^2\rangle}\right)^{1/2},
	\end{align*}
where $A$ is the Stokes operator.

It is shown in \cite{dfj1}, that if the well-recognized power law
	\begin{align}\label{power}
		e_{\kap,2\kap}=\avg{\Sob{P_{2\kap}Q_\kap u}{L^2}^2}\sim\frac{\eta^{2/3}}{\kap^2},
	\end{align}
	holds on
over the {\it inertial range} $[\underline{\kap}_i,\bar{\kap}_i]$ and if
	\begin{align}\label{extraconds}
 \underline{\kap}_i\leq4\kap_\eta,\quad \avg{\Sob{A^{1/2}P_{\underline{\kap}_i} u}{L^2}^2} \ls\avg{\Sob{A^{1/2}Q_{\underline{\kap}_i}u}{L^2}}, \quad G\gs(\bar{\kap}/\kap_0)^2,
	\end{align}
then
	\begin{align}
		\nu^2\kap_0^2\left(\frac{\bar{\kap}}{{\kap_0}}\right)^{-1} G\ls{\langle\Sob{ A^{1/2}u}{L^2}^2\rangle}\ls\nu^2\kap_0^2\left(\frac{\bar{\kap}}{\kap_0}\right) G(\ln G)^{3/2}\label{turb:2d:1}\\
	\nu^2\kap_0^4\left(\frac{\bar{\kap}}{{\kap_0}}\right)^{-3/2}\frac{G^{3/2}}{(\ln G)^{3/2}}\ls{\lb\Sob{Au}{L^2}^2\rb}\ls\nu^2\kap_0^4\left(\frac{\bar{\kap}}{\kap_0}\right)^{3/2}G^{3/2}(\ln G)^{3/4}\;.\label{turb:2d:2}
	\end{align}
 This is to say that \textit{on average} $\Sob{ A^{1/2}u}{L^2}$ is of order $\nu\kap_0G^{1/2}$ on the global attractor.  As in the
 3D case, we can make this precise in terms of probabilities.

First, observe that by the ``time-averaged" Br\'ezis-Gallou\"et inequality (see Proposition \ref{bz:3d})
	\begin{align}
		(\nu\kap_0)^2\lb\Sob{{u}_0}{\W}^2\rb\ls \lb{\Sob{A^{1/2}u_0}{L^2}^2}\rb\left(1+\ln\left({\kap_\s}/{\kap_0}\right)\right)\label{just:2}.\notag
	\end{align}
Hence,
\req{turb:2d:1} and \req{turb:2d:2} imply that
	\begin{align}
		\lb\Sob{u_0}{\W}^2\rb\ls \mathcal{L}G,\notag
	\end{align}
where
	\begin{align}
		\mathcal{L}:=(\bar{\kap}/\kap_0)(\ln G)^{3/2}[1+\ln((\bar{\kap}/\kap_0)^{5/2}G^{1/2}(\ln G)^{3/4})],\notag
	\end{align}
As before, Chebyshev's inequality then implies
	\begin{align}
		\mu\left\{u\in\A:\Sob{u}{\W}\ls \sqrt{\frac{\mathcal{L}}{p}}G^{1/2}\right\}\geq 1-p,
	\end{align}
for any $0<p<1$, provided that either \eqref{power} and \eqref{extraconds} hold.
Therefore, we can conclude by Theorem \ref{gras:peq1} that
	\begin{align}
		\mu\left\{u\in\A:\lr\gs\kap_0^{-1}G^{-1/2}\right\}\geq1-p,
	\end{align} 
where the constant inside depends only on $p$, $\bar{\kap}/\kap_0$, and logarithms of $G$.
Since by \eqref{turb:2d:2}
$$
\lambda_\eta=\left(\frac{\nu^3}{\eta}\right)^{1/6} \le \frac{1}{\ko}\left(\frac{\ko}{\bar\kappa}\right)^{1/4}G^{-1/4}\;,
$$
we have the following
\begin{thm}
Let $\mu$ be a time-invariant measure for a 2D turbulent flow and let $0<p<1$. There exists a set
$S \subset {\mathcal A}$ with $\mu(S) \ge 1-p$ such that
\begin{align*}
		\lam_d(u)\gs_p\kap_0^{-1}(\kap_0\lam_\eta)^{2}\  \mbox{for all}\ u \in S.
	\end{align*}
\end{thm}

\begin{rmk}
There are also 3D versions of a time-averaged Br\'ezis-Gallou\"et inequality, i.e. Proposition \ref{bz:3d}, which accomodate the endpoint cases of the Agmon-type inequality in Proposition \ref{agmon}, namely, $\s=-3/2$ and $\s=-1/2$.  However, neither of these cases fit within our discussion.  Indeed, in the case $\s=-1/2$, one must have some control over the quantity $\Sob{Au}{L^2}/\Sob{A^{1/2}u}{L^2}$, which is not presently known.  On the other hand, although we do have control over the quantity $\Sob{A^{1/2}u}{L^2}/\Sob{u}{L^2}$ in 3D, in this case the Br\'ezis-Gallou\"et inequality will only provide an estimate for the quantity $\Sob{A^{-3/4}u}{L^2}$, which lies outside of the range $\s>-1$ allowed by Theorem \ref{gen:thm3}.  Let us lastly note that if one could control $\Sob{Au}{L^2}/\Sob{A^{1/2}u}{L^2}$, at least on average, then one could argue as before and apply Theorem \ref{gen:thm3} to obtain the  estimate $\lam_d\gs\kap_0^{-1}(\kap_0\lam_\veps)^{7/3}$.
\end{rmk}

\section{Outline of Proofs of Main Theorems}\label{outline}
Following \cite{bis}, our approach is to use a contraction mapping argument.  Fix $0<T\leq \infty$, $\s>-1$, and $\be\geq0$.  Define the spaces
	\begin{align}
		X&:=\{{u}(\ \cdotp)\in C([0,T];V_\s):\noX{{u}}<\infty\}\label{X},\\
		Y&:=\{{u}(\ \cdotp)\in C((0,T];V_{\s+\be}):\noY{{u}}<\infty\}\label{Y},\\
		Z&:=X\cap Y\label{Z},
	\end{align}
where $X,Y,Z$ are equipped with the norms
	\begin{align}
		\noX{u}&:=\frac{\kap_0^{-\s}}{\nu\kap_0}\cdotp\sup_{0\leq t\leq T}\sGev{u(t)}{t}{\s}\label{X1:norm},\\
		\noY{u}&:=\nu^{\be/2}\frac{\kap_0^{-\s}}{\nu\kap_0}\cdotp\sup_{0<t\leq T}(t\wedge(\nu\kap_0^2)^{-1})^{\be/2}\sGev{u(t)}{t}{\s+\be}\label{X2:norm},\\
		\noZ{u}&:=\max\{\noX{u},\noY{u}\},
	\end{align}
and $a\wedge b:=\min\{a,b\}$.  Then $X, Y, Z$ are Banach spaces with $Z\imb X, Y$ continuously.  Observe moreover that these norms are dimensionless.

By the Duhamel principle, the solution $u$ that we seek will be a fixed point of the operator $S$ defined by
	\begin{align}\label{duhamel}
		(Su(\ \cdotp))(t):=\underbrace{e^{-\nu tA}{u}_0+\int_0^te^{-\nu(t-s)A}{\LP f}(s)\ ds}_{\Phi(t)}-\underbrace{\int_0^te^{-\nu(t-s)A}B[u(s),u(s)]\ ds}_{w(t)}.
	\end{align}
In particular, we establish the existence of such a function $u$ in the closed subset $E\sub Z$ given by
	\begin{align}\label{exist:set}
		E:=\{u\in Z:\lVert{u-\Phi}\rVert_{Z}\leq C\},
	\end{align}
for some $C>0$, which satisifes $\noY{\Phi}\leq C$.  To do so, we will invoke the following existence theorem whose proof can be found in \cite{bis}.

\begin{thm}\label{exist:thm}
Suppose that $\Phi\in Z$ and that $\lVert{\Phi}\rVert_{Y}\leq C$ for some $C>0$. 
If $w\in Z$ and $\lVert{w}\rVert_{Z}\leq(1/3)\lVert{v}\rVert_{Y}$ whenever $u\in E$ and $v\in Z$, for $w$ given by either 
	\begin{align}\label{nonlin:term}
		w(t)=\int_0^te^{-\nu(t-s) A}B[u(s),v(s)]\ ds\ \ \text{or}\ \ w(t)=\int_0^te^{-\nu(t-s)A}B[v(s),u(s)]\ ds,
	\end{align}
then there exists a unique $u\in E$ such that 
	\begin{align}
		u=\Phi-\int_0^te^{-\nu(t-s)A}B[u(s),u(s)]\ ds.
	\end{align}
\end{thm}

The hypotheses of Theorem \ref{exist:thm} are verified in Sections \ref{sect:lin} and \ref{proofs}.  In particular, in Section \ref{sect:lin} we show that $\Phi\in Z$ and $\noZ{\Phi}\leq C$ for some $C>0$.  Consequently, this shows that $E$ is nonempty.  We also show in that section that $w\in Z$ whenever $u\in E$ and $v\in Z$.  Finally, in Section \ref{proofs} we deduce sufficient conditions for when $\noZ{w}\leq(1/3)\noY{v}$.   

\section{Estimates with Heat Semigroup}\label{sect:heat}
In this section we list some preliminary estimates.  These estimates concern how the heat kernel, $e^{tA}$, controls the Gevrey multplier, $e^{\lam(t)A^{1/2}}$, where $\al<1$.  The main idea is that the dissipation effect from the heat kernel is stronger than the amplication effect from the Gevrey multiplier.  This idea will also be used to control the nonlinear term.  However, for the nonlinear term one must exploit in a crucial way the Banach algebra structure of $\W$ in the form of a convolution inequality (Proposition \ref{bs:conv:peq1}).  We sketch this below in Proposition \ref{heat:bilin2}.  The proofs of all of these estimates can be found in Sections 5, 6, and 7 of \cite{bis}, where all physical dimensions are normalized.  We have rescaled them here with the relevant physical parameters, and constants as well.  For additional details, see \cite{vthesis}.

\begin{prop}\label{diss:est1}
Let $\nu>0$ and $\be,\lam\geq0$ and let $\s\in\R$.  Then
	\begin{align}
		(\nu t)^{\be/2}\Gev{e^{-\nu tA}u}{\s+\be}\ls C_{\ref{diss:est1}}(\be)\Gev{u}{\s},
	\end{align}
holds for $t>0$, where
	\[
		C_{\ref{diss:est1}}(\be)=\be^{\be/2}.
	\]
\end{prop}

\begin{prop}\label{diss:est2}
Let $\nu>0$ and $\s\in\R$.  Let $\lam:\R_+\goesto\R_+$ be sublinear.  Then 
	\begin{align}
		\Sob{e^{-\nu(t-s)A}u}{\lam(t),\s}\leq C_{\ref{diss:est2}}(s,t)\Sob{e^{-(\nu/2)(t-s)A}u}{\lam(s),\s},
	\end{align}
for all for $0\leq s< t$, where
	\begin{align}\label{c42}
			C_{\ref{diss:est2}}(s,t)=\exp\left(\frac{1}{2\nu}\frac{\lam(t-s)^2}{(t-s)}\right)
	\end{align}
\end{prop}

\begin{rmk}\label{lam}
Observe that Proposition \ref{diss:est2} identifies a suitable sublinear function $\lam(t)$ with which to establish Gevrey regularity, namely, $\lam(t)=\al\sqrt{\nu t}$, for some scalar $\al\geq0$ (see \req{X1:norm} and \req{X2:norm}).  Note that in this case \req{c42} becomes
	\begin{align}
		C_{\ref{diss:est2}}(s,t)=C_{\ref{diss:est2}}(\al)=e^{\al^2/2}.
	\end{align}
For convenience, we set $\al=1$.
\end{rmk}

The following proposition states that the Gevrey norm defined in \req{gev} is a Banach algebra with respect to convolution.

\begin{prop}\label{bs:conv:peq1}
Let $\lam,\gam\geq0$.  Then
	\begin{align}
		\lVert u*v\rVert_{\lam,\gam}\ls
\kap_0^{-\gam}\lVert u\rVert_{\lam,\gam}\Sob{v}{\lam,\gam}.
	\end{align}
\end{prop}

This allows us to establish the following estimate on the nonlinear term.

\begin{prop}\label{heat:bilin2}
Let $\lam,\gam\geq0$.  Then for any $\de\in\R$
	\begin{align}\label{heat:bilin:ineq2}
		\Gev{e^{-\nu tA}B[u,v]}{\de}\ls C_{\ref{heat:bilin2}}(\de,\gam)\kap_0^{1+\de-2\gam}(\nu\kap_0^2 t)^{-\max\{0,(1/2)(1+\de-\gam)\}}\Gev{u}{\gam}\Gev{v}{\gam},
	\end{align}
where
	\[
		C_{\ref{heat:bilin2}}(\de,\gam)=\left({1+\de-\gam}\right)^{\max\{0,(1/2)(1+\de-\gam)\}}
	\]
\end{prop}
\begin{proof}
Let $\al=(1/2)(1+\de-\gam)$.  We estimate as follows
	\begin{align*}
		\Gev{e^{-\nu tA}B[u,v]}{\de}&\leq\sum_{\substack{k=\kap_0k'\\k'\in\Z^n\smod\{0\}}}e^{-\nu t|k|^2}e^{\lam|k|}|k|^{\de }|B[u,v](k)|\\
		&\ls\sum e^{-\nu t|k|^2 }e^{\lam|k|}|k|^{\de+1}(|u|*|v|(k))\\
		&=\lVert e^{-\nu t|k|^2 }|k|^{1+\de-\gam}\rVert_{\ell^\infty}\sum e^{\lam|k|}|k|^{\gam}(|u|*|v|(k))\\
		&\leq\left(\frac{1+\de-\gam}{2 e}\right)^{\max\{0,\al\}}(\nu t)^{-\max\{0,\al\}}\Gev{|u|*|v|}{\gam}\\
		&\ls C_{\ref{heat:bilin2}}(\de,\gam)\kap_0^{1+\de-2\gam}(\nu\kap_0^2 t)^{-\max\{0,\al\}}\Gev{u}{\gam}\Gev{v}{\gam},
	\end{align*}
where in the second inequality we apply \req{leray:ineq} and \req{nonlin:ineq}, while in the last inequality we apply Proposition \ref{bs:conv:peq1}.
\end{proof}

\begin{rmk}\label{algebra}
There is an $\ell^p$-analog of Proposition \ref{heat:bilin2} for $1<p<\infty$.  However, one must restrict the parameter $\gam$ according to the dimension $n$ and index $p$.  This restriction is due to the fact that in general, $\ell^p$ lacks the structure of a Banach algebra for $p>1$ (cf. \cite{bis}).
\end{rmk}

\section{Estimating $\Phi$ and $w$}\label{sect:lin}
First we estimate the  term
	\begin{align}\label{linear}
		\Phi(t):=e^{-\nu tA}{u}_0+\int_0^te^{-\nu(t-s)A}{ f}(s)\ ds
	\end{align}
in order to show that $\Phi\in Z$ (see \req{Z} and \req{duhamel}).

\begin{lem}\label{linear:lem}
Let $1<q\leq\infty$ and $1/q'=1-1/q$.  Let $\s\in\R$ and $M$ be given as in \req{nd:M}.
Then for $0\leq\be<2/q'$ and $0<T\leq T_f$,
	\begin{enumerate}[(i)]
		\item $\noX{{\Phi}}\ls C_{\ref{linear:lem}}^{(i)}(q)M$, where
			\begin{align*}
				C_{\ref{linear:lem}}^{(i)}(q)=(1/q')^{1/q'}.
			\end{align*}
		\item $\noY{{\Phi}}\ls C_{\ref{linear:lem}}^{(ii)}(q,\be,\lam)M$, where
			\begin{align*}
				C_{\ref{linear:lem}}^{(ii)}(q,\be)=&C_{\ref{diss:est1}}(\be)C_{\ref{beta:int}}(\be q'/2,0)^{1/q'}(q')^{\be/2}.
			\end{align*}
		\item $(\nu t)^{\be/2}\sGev{{\Phi}(t)}{t}{\s+\be}\leq C(t)$, for $0<t\leq T_f$, with $\lim_{t\goesto0^+}C(t)=0$, if $\be>0$.
	\end{enumerate}
\end{lem}

\begin{proof}
Fix $T\leq T_f$ and let $0\leq t\leq T$.  Observe that by \req{leray:ineq}
	\begin{align*}
		\sGev{\Phi(t)}{t}{\s}\ls \underbrace{\sGev{e^{-\nu tA}u_0}{t}{\s}}_{(A)}+\underbrace{\int_0^t\sGev{e^{-\nu(t-s)A}{ f}(s)}{t}{\s}\ ds}_{(B)}.
	\end{align*}
We estimate $(A)$ by applying Proposition \ref{diss:est2} with $s=0$ and using the fact that $e^{-\nu tA}$ is a contractive semigroup for $t>0$ so that
	\begin{align}\label{est:A}
		\sGev{e^{-\nu tA}u_0}{t}{\s}\ls \Sob{e^{-(\nu/2)tA}u_0}{\s}\leq\Sob{u_0}{\s}.
	\end{align}

Now we estimate $(B)$.  Observe again that by contractivity and Proposition \ref{diss:est2}
	\begin{align}
		\sGev{e^{-\nu(t-s)A} f(s)}{t}{\s}&\ls \sGev{e^{-(\nu/2)(t-s)A}f(s)}{s}{\s}\label{est:B1}.
	\end{align}
Suppose $1<q<\infty$.  Integrating both sides of \req{est:B1} and applying the H\"older inequality gives
	\begin{align}\label{est:B}
		\int_0^t&\sGev{e^{-\nu(t-s)A}f(s)}{t}{\s}\ ds\notag\\
		&\ls (2/q')^{1/q'}(\nu\kap_0^2)^{-1}\left(\nu\kap_0^2\int_0^{T_f}\sGev{f(s)}{s}{\s}^q\ {ds}\right)^{1/q}
	\end{align}
where $q,q'$ are H\"older conjugates. 
Adding \req{est:A}, \req{est:B}, normalizing physical dimensions, then taking the supremum proves $(i)$.  For $q=\infty$, make an $L^1$-$L^\infty$ H\"older estimate in \req{est:B1} instead.

To prove $(ii)$, instead let $0<t\leq T$.  Observe that
	\begin{align}\label{G:est1}
		\sGev{\Phi(t)}{t}{\s+\be}\ls \underbrace{\sGev{e^{-\nu tA}u_0}{t}{\s+\be}}_{(A')}+\underbrace{\int_0^t\sGev{e^{-\nu(t-s)A}f(s)}{t}{\s+\be}\ ds}_{(B')}.
	\end{align}
We estimate $(A')$ as
	\begin{align}
		\sGev{e^{-\nu tA}u_0}{t}{\s+\be}&\ls \Sob{e^{-(\nu/2)tA}u_0}{\s+\be}\notag\\
		&\ls C_{\ref{diss:est1}}(\nu t/2)^{-\be/2}\Sob{u_0}{\s}\notag\\
		&\leq C_{\ref{diss:est1}}(\nu /2)^{-\be/2}(t\wedge((\nu\kap_0^2)/2)^{-1})^{-\be/2}\Sob{u_0}{\s}\label{est:A'}.
	\end{align}
Similarly, assuming $1<q<\infty$, we can estimate $(B')$ as
	\begin{align}
		\sGev{e^{-\nu(t-s)A} f(s)}{t}{\s+\be}
		&\ls C_{\ref{diss:est1}}e^{-(\nu/q')(t-s)\kap_0^2}(\nu(t-s)/q')^{-\be/2}\sGev{f(s)}{s}{\s}\label{est:B'}.
	\end{align}
Now integrate both sides of \req{est:B'}, apply the H\"older inequality, then Proposition \ref{beta:int} to obtain
	\begin{align}
		(B')	\ls &C_{\ref{diss:est1}}\int_0^t\frac{e^{-(\nu/q')(t-s)\kap_0^2}}{(\nu(t-s)/q')^{\be/2}}\sGev{f(s)}{s}{\s}\ ds\label{holder}\\
		\leq &C_{\ref{diss:est1}} C_{\ref{beta:int}}^{1/q'}\cdotp(\nu/q')^{-\be/2}(t\wedge(\nu\kap_0^2)^{-1})^{1/q'-\be/2}(\nu\kap_0^2)^{-1/q}\frac{\kap_0^\s}{\nu^{-2}\kap_0^{-3}}M_f\label{T:split},
	\end{align}
where
	\begin{align}\label{euler:beta}
	C_{\ref{beta:int}}(c,d)=\mathcal{B}(1-c,1-d)=\int_0^1t^{-c}(1-t)^{-d}\ dt.
	\end{align}
An elementary calculation shows that $\mathcal{B}(1-c,1)=\frac{1}{1-c}$, which in particular implies that
	\begin{align} 
		C_{\ref{beta:int}}((\be q'/2),0)>1\label{beta:const}.
	\end{align}
Also, observe that for any $c\geq1$
	\begin{align}\label{min:simp}
		(t\wedge(\nu\kap_0^2)^{-1})\leq(t\wedge((\nu\kap_0^2)/c)^{-1})\leq c(t\wedge(\nu\kap_0^2)^{-1}).
	\end{align}
Therefore, by adding \req{est:A'} and \req{T:split} , then applying \req{beta:const} and \req{min:simp} we obtain
	\begin{align}\label{ii:units}
		\nu^{\be/2}\frac{\kap_0^{-\s}}{\nu\kap_0}&(t\wedge(\nu\kap_0^2)^{-1})^{\be/2}\sGev{\Phi(t)}{t}{\s+\be}\notag\\
	&\ls C_{\ref{diss:est1}}C_{\ref{beta:int}}^{1/q'}\left(\frac{\kap_0^{-\s}}{\nu\kap_0}\Sob{u_0}{\s}+(t\wedge(\nu\kap_0^2)^{-1})^{1/q'}(\nu\kap_0^2)^{1/q'}M_f\right).
	\end{align}
Using the fact that $(t\wedge(\nu\kap_0^2)^{-1})\leq (\nu\kap_0^2)^{-1}$, then taking the supremum over $0<t\leq T$ completes the proof of $(ii)$ for $1<q<\infty$.

If $q=\infty$, then instead make an $L^1$-$L^\infty$ H\"older estimate in \req{holder}, so that \req{T:split} becomes
	\begin{align*}
		\int_0^t&\sGev{e^{-\nu(t-s)A}f(s)}{t}{\s+\be}\ ds\\
				&\ls C_{\ref{diss:est1}} C_{\ref{beta:int}}\cdotp(\nu/2)^{-\be/2}(t\wedge(\nu\kap_0^2/2)^{-1})^{1-\be/2}\frac{\kap_0^\s}{\nu^{-2}\kap_0^{-3}}{M}_f,
	\end{align*} 
Then apply \req{min:simp} again.

Finally, we prove $(iii)$.  By Proposition \ref{diss:est2} we have
\begin{align*}
		(\nu t)^{\be/2}&\sGev{\Phi(t)}{t}{\s+\be}\\
			&\ls (\nu t)^{\be/2}\Sob{e^{-(\nu/2)tA}u_0}{\s+\be}+(\nu t)^{\be/2}\left(\int_0^t\Sob{e^{-(\nu/2)(t-s)A}f(s)}{\sqrt{\nu s},\s+\be}\ ds\right).
	\end{align*}
Now consider the projection $P_{{\kap}}$ onto modes $|k|\leq{\kap}/\kap_0$ with $Q_\kap=I-P_\kap$.  Observe that
	\begin{align*}
		\Sob{e^{-(\nu/2)tA}u_0}{\s+\be}&\leq \Sob{e^{-(\nu/2)tA}Q_{\kap}u_0}{\s+\be}+\Sob{e^{-(\nu/2)tA}P_\kap u_0}{\s+\be}\\
		&\ls C_{\ref{diss:est1}}(\nu t)^{-\be/2}\Sob{Q_\kap u_0}{\s}+\Sob{P_\kap u_0}{\s+\be}.
	\end{align*}
Similarly
	\begin{align*}
		(\nu t)^{\be/2}\Sob{e^{-(\nu/2)(t-s)A}f(s)}{\sqrt{\nu s},\s+\be}&\ls C_{\ref{diss:est1}}\Sob{Q_\kap f(s)}{\sqrt{\nu s},\s}+(\nu t)^{\be/2}\Sob{P_\kap f(s)}{\sqrt{\nu s},\s+\be}.
	\end{align*}
Since $\kap$ is arbitrary, sending $t\goesto0^+$ completes the proof.

\end{proof}

\begin{cor}\label{linear:cor}
Under the same hypotheses as Lemma \ref{linear:lem}, suppose moreover that
	\begin{align}\label{small:data:cor}
		M_0\ls ({T}\nu\kap_0^2)^{1/q'}M_f
	\end{align}
where $T\leq T_f$.  Then
	\begin{enumerate}[(i)]
		\item $\noX{{\Phi}}\ls C_{\ref{linear:lem}}^{(i)}(q)(T\nu\kap_0^2)^{1/q'}M_f$,
		\item $\noY{{\Phi}}\ls C_{\ref{linear:lem}}^{(ii)}(q,\be)({T}\nu\kap_0^2)^{1/q'}M_f$.
	\end{enumerate}
\end{cor}

\begin{proof}
First, recall \req{est:B1} from the proof of Lemma \ref{linear:lem} (i)
	\begin{align}\label{est:B2}
		\sGev{e^{-\nu(t-s)A}f(s)}{t}{\s}&\ls e^{-(\nu/2)(t-s)\kap_0^2}\sGev{f(s)}{s}{\s}.
	\end{align}
Since $s\leq t$, we have $e^{-(\nu/4)(t-s)\kap_0^2}\leq1$.  Thus, by integrating \req{est:B2} and applying H\"older's inequality
	\begin{align}
		\frac{\kap_0^{-\s}}{\nu\kap_0}\int_0^t&\sGev{e^{-\nu(t-s)A}f(s)}{t}{\s}\ ds\ls  (1/q')^{1/q'}(T\nu\kap_0^2)^{1/q'}M_f.
	\end{align}
After normalizing, we add \req{est:A} to finish the proof of (i). 

On the other hand, recall \req{ii:units} in the proof of Lemma \ref{linear:lem} (ii), which we  rewrite as
	\begin{align}\notag
		\nu^{\be/2}\frac{\kap_0^{-\s}}{\nu\kap_0}(t\wedge(\nu\kap_0^2)^{-1})^{\be/2}&\sGev{{\Phi}(t)}{t}{\s+\be}\\
			&\ls C_{\ref{linear:lem}}^{(ii)}\left(M_0+(T\wedge(\nu\kap_0^2)^{-1})^{1/q'}(\nu\kap_0^2)^{1/q'}M_f\right),
	\end{align}
for all $0< t\leq T$.  Therefore, \req{small:data:cor} and the fact that  $(T\wedge(\nu\kap_0^2)^{-1})\leq T$ proves (ii).
\end{proof}

The following lemma provides the necessary estimate for
	\begin{align}\label{nonlin}
		w(t):=\int_0^te^{-\nu(t-s)A}B[u(s),v(s)]\ ds.
	\end{align}

\begin{lem}\label{nonlin:peq1}
Let $\s>-1$.  Let $0\leq\be<1$ such that $\gam=\s+\be\geq0$.
Then
	\begin{align*}
		\noZ{w}\ls C_{\ref{nonlin:peq1}}(\be)(\nu\kap_0^2)^{(1-\be)/2}(T\wedge(\nu\kap_0^2)^{-1})^{(1-\be)/2}\noY{u}\noY{v},
	\end{align*}
where
	\begin{align*}
		C_{\ref{nonlin:peq1}}(\be)=&\max\{C_{\ref{beta:int}}((1-\be)/2,\be),C_{\ref{beta:int}}(1/2,\be)\}
	\end{align*}
\end{lem}

Its proof follows exactly that of Proposition 8.5 in \cite{bis}.  For additional details see \cite{vthesis}.

\section{Proofs of Main Theorems}\label{proofs}

\begin{proof}[Proof of Theorem \ref{gen:thm2}]
Let $\s>-1$ and $\s_-:=\max\{-\s,0\}$.  Define $\be=\be(\s,q)$ by
	\begin{align}\label{beta}
		\be:=\begin{cases}
			2\s_-/q',&1<q\leq2\\
			\s_-,&2\leq q\leq\infty.
	\end{cases}
	\end{align}
Observe that $0\leq\be<\min\{2/q',1\}$ holds for all $1< q\leq\infty$.  Let $X,Y, Z$ be given by \req{X}, \req{Y}, \req{Z} respectively.  Let ${\Phi}$ be defined by \req{linear}.  Then by Lemma \ref{linear:lem}, we have $\Phi\in Z$ and
	\begin{align}\notag
		\noY{\Phi}\leq C_{\ref{linear:lem}}^{(ii)}M.
	\end{align}
Thus, the set $E\sub X$ given by \req{exist:set} becomes
	\begin{align}\notag
		E=\{u\in Z:\noZ{u-\Phi}\leq C_{\ref{linear:lem}}^{(ii)}M\}.
	\end{align}
Obviously, Lemma \ref{nonlin:peq1} implies that $w\in Z$ whenever $u\in E$ and $v\in Z$, where $w$ is given by \req{nonlin}.  Hence, by Theorem \ref{exist:thm}, it suffices to show that $\noZ{w}\leq(1/3)\noY{v}$, whenever $u\in E$ and $v\in Z$.  We determine sufficient conditions for this to hold.

By Lemma \ref{nonlin:peq1} we have
	\begin{align}\notag
		\noZ{w}\ls C_{\ref{nonlin:peq1}}(\be)(\nu\kap_0^2)^{(1-\be)/2}(T\wedge(\nu\kap_0^2)^{-1})^{(1-\be)/2}\noY{u}\noY{v},
	\end{align}
for any $u,v\in Y$, and in particular, for any $u\in E$.  By definition of $E$, $\noY{u}\leq 2C_{\ref{linear:lem}}^{(ii)}M$ whenever $u\in E$, so that
	\begin{align}\notag
			\noZ{w}\ls C_{\ref{nonlin:peq1}}C_{\ref{linear:lem}}^{(ii)}(\nu\kap_0^2)^{(1-\be)/2}(T\wedge(\nu\kap_0^2)^{-1})^{(1-\be)/2}M\noY{v}.
	\end{align}
Thus, to satisfy $\noZ{w}\leq(1/3)\noY{v}$ it suffices to have
	\begin{align}\notag
		C\cdotp C_{\ref{nonlin:peq1}}C_{\ref{linear:lem}}^{(ii)}(\nu\kap_0^2)^{(1-\be)/2}T^{(1-\be)/2}M\leq 1/3,
	\end{align}
for some sufficiently large absolute constant $C>0$.
In other words, if
	\begin{align}
		T^*= (C^*)^{2/(1-\be)}(\nu\kap_0^2)^{-1}M^{-2/(1-\be)},\notag
	\end{align}
where $C^*$ is given by
	\begin{align}\label{cstar}
		C^*:=(1/(3C))(C_{\ref{nonlin:peq1}}C_{\ref{linear:lem}}^{(ii)})^{-1},
	\end{align}
for some large $C>0$, then there exists a unique $u\in E$ such that $u=\Phi-w$, whose radius of analyticity at time $T^*$ is at least
	\begin{align}\notag
		\lr\gs \kap_0^{-1}M^{-1/(1-\be)}.
	\end{align}
In particular, since $u\in X$ with $\lam(s)=\sqrt{\nu s}$, $u$  is Gevrey regular.

On the other hand, if we instead assume that
	\begin{align}
		M\ls C^*\notag,
	\end{align}
then the solution $u$ exists up to time $T^*=T_f$.  Hence, $\lr\gs \sqrt{\nu T_f}$.

The proof that $u$ is also a weak solution follows exactly as in \cite{bis} (pp. 1184-85).  This completes the proof.
\end{proof}

\begin{proof}[Proof of Theorems \ref{gras:peq1} and \ref{gras:peq1:3d}]
Let $M_0$ and $M_f$ be given by \req{M0} and \req{M:data}, respectively.  Let $\be$ be given by \req{beta}.  Assume that
	\begin{align}\label{id:q}
		M_0\ls C_*M_f^{(1-\be)/(1-\be+2/q')},
	\end{align}
where 
	\begin{align}\label{cstar2}
		C_*=(C^*)^{(2/q')/(1-\be+2/q')},
	\end{align}
and $C^*$ is given by \req{cstar}.
Let 
	\begin{align}\label{tstar2}
		T^*=(C_*)^{q'}(\nu\kap_0^2)^{-1}M_f^{-2/(1-\be+2/q')}.
	\end{align}

Now let $E$ be given by
	\begin{align}\notag
		E=\{u\in Z:\noZ{u-\Phi}\leq C_{\ref{linear:lem}}^{(ii)}C_*M_f^{(1-\be)/(1-\be+2/q')}\}.
	\end{align}
Since \req{id:q} holds, by Corollary \ref{linear:cor} (with $T=T^*$), we know $\Phi\in Z$ such that
	\begin{align}\notag
		 \noY{\Phi}\ls C_{\ref{linear:lem}}^{(ii)}C_*M_f^{(1-\be)/(1-\be+2/q')}.
	\end{align}
We can now verify the condition $\noZ{w}\leq(1/3)\noY{v}$ for $u\in E$ and $v\in X$, directly.  Indeed, proceeding as in the proof of Theorem \ref{gen:thm2}, we know that by Lemma \ref{nonlin:peq1}
	\begin{align}\notag
		\noZ{w}\ls C_{\ref{nonlin:peq1}}(\nu\kap_0^2)^{(1-\be)/2}T^{(1-\be)/2}\noY{u}\noY{v},
	\end{align}
for all $T\leq T^*$, whenever $u,v\in Y$.  Now observe that for  $u\in E$, we have $\noY{u}\leq 2C_{\ref{linear:lem}}^{(ii)}C_*M_f^{(1-\be)/(1-\be+2/q')}$.  Hence, by definition of \req{cstar2} and \req{tstar2}
	\begin{align}
		\noZ{w}&\leq C\cdotp C_{\ref{nonlin:peq1}}(\nu\kap_0^2)^{(1-\be)/2}(T^*)^{(1-\be)/2}C_{\ref{linear:lem}}^{(ii)}C_*M_f^{(1-\be)/(1-\be+2/q')}\noY{v}\notag\\
		 		&\leq C\cdotp C_{\ref{nonlin:peq1}}C^*C_{\ref{linear:lem}}^{(ii)}\noY{v}\notag\\
				&=(1/3)\noY{v},\notag
	\end{align}
where $C>0$ is some large absolute constant.

Thus, Theorem \ref{exist:thm} furnishes a unique $u\in E$ such that the radius of analyticity at time $T^*$ satisfies
	\begin{align}
		\lr\gs\kap_0^{-1}M_f^{-1/(1-\be+2/q')}.
	\end{align}
As before, $u$ is also a weak solution.  

For Theorem \ref{gras:peq1}, set $\s=0$ and $q=2$, so that $\be=0$.  Then
	\begin{align}\label{thm:peq1:2d}
		\lr\gs\kap_0^{-1}M_f^{-1/2},
	\end{align}
provided that
	\begin{align}\label{cond:peq1:2d}
		M_0\ls M_f^{1/2}.
	\end{align}

For Theorem \ref{gras:peq1:3d}, set $\s=-3/4$ and $q=59/49$, so that $\be=15/59$.  Then
	\begin{align}\label{thm:peq1:3d}
		\lr\gs\kap_0^{-1}M_f^{-59/64},
	\end{align}
provided that
	\begin{align}\label{cond:peq1:3d}
		M_0\ls M_f^{11/16}
	\end{align}
Finally, let $\tau:=(\nu\kap_0^2)^{-1}$ and $\lam_f:=\kap_0^{-1}$.  Observe that for any $0\leq s\leq \tau$
	\[
		\sqrt{\nu s}\leq \sqrt{\nu(\nu\kap_0^2)^{-1}}=\kap_0^{-1}.
	\]
Applying Proposition \ref{M:gras} with this choice of $\tau$ and $\lam_f$ to \req{thm:peq1:2d}-\req{cond:peq1:3d} establishes the desired lower bound in Theorems \ref{gras:peq1} and \ref{gras:peq1:3d}.  Since $V_0\sub\ell^2$ and $C([0,T^*];V_0)\sub L^\infty([0,T^*];\ell^2)$, uniqueness of $u$ as a weak solution follows from a criterion of Lions (cf. \cite{temamnse} pp. 298-99).
\end{proof}

We have, in fact, just proven the following, more general theorem.

\begin{thm}\label{gen:thm3}
Let $1< q\leq\infty$ with $1/q'=1-1/q$, and $\s>-1$.  Let $\be$ be given by \req{beta} and $M_0,M_f$ be given by \req{M0} and \req{M:data}, respectively.  Suppose that $f$ satisfies $M_f<\infty$.  If
	\begin{align}\label{small:Mf:1}
		M_0\ls M_f^{(1-\be)/(1-\be+2/q')},
	\end{align}
then there exists $T^*<T_f$ and mild solution $u\in C([0,T^*];V_\s)$ to \req{wv:sys} such that $u$ is also a Gevrey regular weak solution, with radius of analyticity at time $T^*$ satisfying
	\begin{align}\label{small:avgMf:1}
		\lr\gs\kap_0^{-1}{M_f^{-1/(1-\be+2/q')}}.
	\end{align}
\end{thm}

\begin{rmk}\label{rmk:thm2}
Observe that Theorem \ref{gen:thm3} gives some freedom over the assumption on $M_0$.  For instance, if $\s\geq0$ and $1\leq q'<2$, then $\lr$ at time $T^*$, given by \req{tstar2}, will satisfy the improved estimate
	\begin{align}
		\lr\gs\kap_0^{-1}M_f^{-1/(1+2/q')},
	\end{align}
provided that
	\begin{align}\label{gen:assum}
		M_0\ls M_f^{1/(1+2/q')}.
	\end{align}
If $f$ is time-independent with finitely many modes, then by Proposition \ref{M:gras} we can replace $M_f$ with $G$.  It would be interesting to know  if \req{gen:assum} can be established \textit{on average} on the global attractor in 2D in the spirit of \cite{dfj1}, for $\s=0$ and some $1\leq q'<2$, for example, without invoking Br\'ezis-Gallou\"et and the estimates established by \cite{dfj1}.  Indeed, if $q'=1$, then Theorem \ref{gen:thm3} yields the estimate $\lam_a\gs G^{-1/3}$, which would recover the estimate for $\lam_d$ predicted by the Kraichnan theory of 2D turbulence (see \cite{kraich}).
\end{rmk}

\section{Appendix}\label{appendix}
We require the following elementary inequality.

\begin{prop}\label{beta:int}
Let $b\geq0$ and $0\leq c,d<1$.  Then for all $t>0$
	\begin{align}\label{beta:eq}
	\int_0^t\frac{e^{-b(t-s)}}{(t-s)^c(s\wedge b^{-1})^{d}}\ ds\leq C_{\ref{beta:int}}(c,d)(t\wedge b^{-1})^{1-c-d},
	\end{align}
where $C_{\ref{beta:int}}(c,d)=\max\{\mathcal{B}(1-c,1-d)$$,\Gam(1-c)\}$, where $\Gam$ is the gamma function and $\mathcal{B}$ is the beta function.
\end{prop}
\begin{proof} Firstly, if $b=0$, then set $(x\wedge b^{-1})=x$.

Observe that
	\[
	\int_0^t\frac{e^{-b(t-s)}}{(t-s)^c(s\wedge b^{-1})^{d}}\ ds\leq\int_0^t\frac{1}{(t-s)^cs^{d}}\ ds=t^{-c-d}\int_0^t\left(1-\frac{s}t\right)^{-c}\left(\frac{s}t\right)^{-d}\ ds.
	\]
Making the change of variables $\s=s/t$ and assuming that $bt\leq1$, we have
	\begin{align*}
	t^{-c-d}\int_0^t(1-\frac{s}t)^{-c}(\frac{s}t)^{-d}\ ds&\leq t^{1-c-d}\int_0^1(1-\s)^{-c}\s^{-d}\ d\s\\
	&=t^{1-c-d}\int_0^1(1-\s)^{(1-c)-1}\s^{(1-d)-1}\ d\s\\
	&=\mathcal{B}(1-c,1-d)(t\wedge b^{-1})^{1-c-d},
	\end{align*}
where $\mathcal{B}$ is given by \req{euler:beta}.

On the other hand, if $bt>1$
, observe that
	\begin{align*}
	\int_0^t\frac{e^{-b(t-s)}}{(t-s)^c(s\wedge b^{-1})^{d}}\ ds&=b^{d}\int_0^t(t-s)^{-c}e^{-b(t-s)}\ ds\\
	&=b^{d}\int_0^t(t-s)^{-c}e^{-b(t-s)}\ ds\\
	&=b^{d-1}\frac{1}{b^{-c}}\int_0^{bt}\s^{-c}e^{-\s}\ d\s\\
	&\leq(b^{-1})^{1-c-d}\int_0^{\infty}\s^{(1-c)-1}e^{-\s}\ d\s\\
	&=\Gam(1-c)(t\wedge b^{-1})^{1-c-d}.
	\end{align*}
\end{proof}

Now, we prove Proposition \ref{M:gras}, which establishes the equivalency (up to a constant) of $M_f$ (see \req{M:data}) and the Grashof number, $G$ (see \req{gras}).

\begin{prop}\label{M:gras}
Let $n>1$.  Suppose that $f$ is time-independent and satisfies $f=P_{\bar{\kap}}f$.  Let $\lam_f$ be given such that
	\begin{align}
		\sup_{|y|\leq\lam_f}\Sob{f(\ \cdotp+iy)}{L^2}<\infty,
	\end{align}  
and $\lam:\R^+\goesto\R^+$ satisfy $\lam(s)\leq\lam_f$ whenever $0\leq s\leq \tau$, for some $\tau>0$.  Then
	\begin{align}\label{simG}
		M_f\sim_{\s,\bar{\kap},\lam_f,\tau}G,
	\end{align}
where the constants are explicitly identified in \req{const:1} and \req{const:2}.
\end{prop}

\begin{proof}[Proof of Proposition \ref{M:gras}]
Let $z=x+iy$ with $x\in[0,L]^n$ and $|y|\leq\lam(s)$.  Then we can write $f(z)=\sum_{\lv k\rv\leq\bar{\kap}/\kap_0}\hat{f}(k)e^{i\kap_0k\cdotp z}$.  Observe that since $\kap_0=2\pi/L$
	\begin{align*}
		\lVert f(\ \cdotp +iy)\rVert_{L^2}^2&=\sum_{\lv k\rv,\lv \ell\rv\leq\bar{\kap}/\kap_0}\hat{f}(k)\overline{\hat{f}(\ell)}e^{\kap_0(k+\ell)\cdotp y}\int_{[0,L]^n}e^{i\kap_0(k-\ell)\cdotp x}\ dx\\
			&=(2\pi)^{n}\kap_0^{-n}\sum_{|k|\leq\bar{\kap}/\kap_0}|\hat{f}(k)|^2e^{2\kap_0k\cdotp y}.
	\end{align*}
This implies that
	\begin{align}\notag
			e^{-2\bar{\kap}\lam_f}\kap_0^{-n/2}\Sob{e^{\lam(s)A^{1/2}}f}{\ell^2}\ls \lVert f(\ \cdotp +iy)\rVert_{L^2}\ls\kap_0^{-n/2}\Sob{e^{\lam(s)A^{1/2}}f}{\ell^2},
	\end{align}
for all $|y|\leq\lam(s)$.  Hence
	\begin{align}\notag
		\frac{1}{\nu^2\kap_0^3}\Sob{e^{\lam(s)A^{1/2}}f}{\ell^2}\sim_{\bar{\kap},\lam_f}\frac{\kap_0^{n/2}}{\nu^2\kap_0^3}{\sup_{\lv y\rv\leq\lam({s})}\lVert f(\ \cdotp+iy)\rVert_{L^2}}.
	\end{align}
Now recall the following elementary facts:
	\begin{itemize}
		\item $\lVert f\rVert_{\ell^q}\leq\lVert f\rVert_{\ell^p}\ls_{p,q,\bar{\kap}}\lVert f\rVert_{\ell^q}$ for $1\leq p<q<\infty$;
		\item $\lVert f\rVert_{\ell^p}\leq\kap_0^{-\s}\Sob{f}{\s}\leq \left(\frac{\bar{\kap}}{\kap_0}\right)^\s\lVert f\rVert_{\ell^p}$ for $1\leq p\leq\infty$
	\end{itemize}
These imply that
	\begin{align}\label{sim1}
		\frac{\kap_0^{-\s}}{\nu^2\kap_0^3}\Sob{f}{\lam(s),\s}\sim_{\s,\bar{\kap},\lam_f} \frac{\kap_0^{n/2}}{\nu^2\kap_0^3}{\lVert f(\ \cdotp+iy)\rVert_{L^2}},
	\end{align}
for all $|y|\leq\lam(s)$.  Obviously, if we set $y=0$, then by the definition of the Grashof number (see \req{gras}), we get
	\begin{align}\notag
		\frac{\kap_0^{-\s}}{\nu^2\kap_0^3}\sup_{0\leq s\leq \tau}\Sob{f}{\lam(s),\s}\sim_{\s,\bar{\kap},\lam_f} G.
	\end{align}
On the other hand, for $1\leq q<\infty$, if we take the $L^q((0,\tau), ds/(\nu\kap_0^2)^{-1})$ norm of \req{sim1},
then
	\begin{align}\label{sim2}
		M_f\sim_{\s,\bar{\kap},\lam_f,\tau}\frac{\kap_0^{n/2}}{\nu^2\kap_0^3}
		\Sob{f(\ \cdotp+iy)}{L^2},
	\end{align}
for all $|y|\leq\lam(s)$.  
Thus, by setting $y=0$ in \req{sim2} and by definition of \req{M:data}, we deduce that
	\begin{align}\notag
		M_f\sim_{\s,\bar{\kap},\lam_f,\tau}G.
	\end{align}
In particular, we have
	\begin{align}\label{const:1}
		C_{\lam_f,\bar{\kap},n}M_f\leq (\nu\kap_0^2\tau)^{1/q} G\leq C_nM_f,
	\end{align}
where $C_n:=(2\pi)^n$ and
	\begin{align}\label{const:2}
		C_{\lam_f,\bar{\kap},n}&:=(2\pi)^{-n}\left(\sum_{|k|\leq\bar{\kap}}1\right)^{-1/2}e^{-2\lam_f\bar{\kap}}\left(\frac{\kap_0}{\bar{\kap}}\right)^\s.
	\end{align}
\end{proof}

We also made use of a ``time-averaged" Br\'ezis-Gallou\"et-type inequality in 2D and an Agmon-type inequality in 3D.  Our proof of the Br\'ezis-Gallou\"et-type inequality mimics that in \cite{doergibb} with an additional step to accomodate time-averages (see \req{time:avg:step}).  The proof of the Agmon-type inequality follows along the same lines.  We supply both of them here for the sake of completion.

\begin{prop}\label{bz:3d}
Let $L>0$ and $\Om=[0,L]^2$.  Let $\A$ be the global attractor of \req{wv:sys} with time-independent forcing $f$ satisfying $P_{\bar{\kap}}f=f$.  Then there exists an absolute constant $C>0$ such that
	\begin{align}
		(\nu\kap_0)^2\lb\Sob{u}{\W}^2\rb\leq C\lb\Sob{A^{1/2}u}{L^2(\Om)}^2\rb\left[1+\ln\left(\kap_0^{-2}\frac{\lb\Sob{Au}{L^2(\Om)}^2\rb}{\lb\Sob{A^{1/2}u}{L^2(\Om)}^2\rb}\right)\right],
	\end{align}
for all $u\in\A$, where $A$ is the Stokes operator, and $\lb\ \cdotp\rb$ denotes an ensemble average in the sense of \req{time:avg}.
\end{prop}

\begin{proof}
Let $u_k:=|\hat{u}(k)|$ for all $k\in\Z^n$.  Fix $\lam>0$ to be chosen later  Observe that
	\begin{align}
		\sum_{k\in\Z^d} u_k=\underbrace{\sum_{|k|\leq\lam}|k|^{-1}|k|u_k}_{A}+\underbrace{\sum_{|k|>\lam}|k|^{-2}|k|^2u_k}_{B}.\notag
	\end{align}
Estimate $A$ with Cauchy-Schwarz to get
	\begin{align}
		A\leq\left(\sum_{|k|\leq\lam}|k|^{-2}\right)^{1/2}\left(\sum_{|k|\leq\lam} |k|^2u_k^2\right)^{1/2}.\notag
	\end{align}
Observe that
	\begin{align}
		\sum_{|k|\leq\lam}|k|^{-2}\leq C\int_1^\lam r^{-1}\ dr=C\log\lam.\notag
	\end{align}
On the other hand, we estimate $B$ as follows
	\begin{align}
		B\leq\left(\sum_{|k|>\lam}|k|^{-4}\right)^{1/2}\left(\sum_{|k|>\lam}|k|^{4}u_k^2\right)^{1/2}.\notag
	\end{align}
Observe that
	\begin{align}
		\sum_{|k|>\lam}|k|^{-4}\leq C\int_{\lam}^\infty r^{-3}\ dr=\frac{C}{2}\lam^{-2}.\notag
	\end{align}
Combining $A$ and $B$, so far we have
	\begin{align}\notag
		\Sob{\ve{u}}{\ell^1}\leq &C(\log\lam)\Sob{|\ \cdotp|\ve{u}}{\ell^2}+\frac{C}{2}\lam^{-2}\Sob{|\ \cdotp|^2\ve{u}}{\ell^2},
	\end{align}
An elementary calculation gives
	\begin{align}
		\Sob{\ve{u}}{\ell^1}^2\leq &2C^2(\log\lam)^2\Sob{|\ \cdotp|\ve{u}}{\ell^2}+\frac{C^2}2\lam^{-4}\Sob{|\ \cdotp|^2\ve{u}}{\ell^2}^2.\notag
	\end{align}
Taking time-averages, monotonicity and linearity of generalized Banach limits imply
	\begin{align}\label{time:avg:step}
		\lb\Sob{\ve{u}}{\ell^1}^2\rb\leq &C(\log\lam)\lb\Sob{|\ \cdotp|\ve{u}}{\ell^2}^2\rb+\frac{C}{2}\lam^{-2}\lb\Sob{|\ \cdotp|^2\ve{u}}{\ell^2}^2\rb,
	\end{align}
Now choose $\lam$ such that
	\begin{align}
		\lam^{-2}=\frac{\lb\Sob{|\ \cdotp|\ve{u}}{\ell^2}^2\rb}{\lb\Sob{|\ \cdotp|^2\ve{u}}{\ell^2}^2\rb}.\notag
	\end{align}
Observe that $\lam\geq1$.  Therefore, for some absolute constant $C>0$,
	\begin{align}
		\lb\Sob{\ve{u}}{\ell^1}^2\rb\leq C\lb\Sob{\ve{u}}{\ell^2}^2\rb\left[1+\ln\left(\frac{\lb\Sob{|\ \cdotp|^2\ve{u}}{\ell^2}^2\rb}{\lb\Sob{|\ \cdotp|\ve{u}}{\ell^2}^2\rb}\right)\right].\notag
	\end{align}
Rescaling with physical units and applying Parseval's identity completes the proof.
\end{proof}

\begin{prop}\label{agmon}
Let $\Om:=[0,L]^3$.  Suppose that  $u\in H^1(\Om)$ has mean zero.  Then
	\begin{align}
		(\nu\kap_0)\Sob{A^{\s/2}{u}}{\W}\leq C_{\ref{agmon}}\Sob{{u}}{L^2}^{-(\s+1/2)}\Sob{A^{1/2} u}{L^2}^{\s+3/2},
	\end{align}
for any $-3/2<\s<-1/2$, where $A$ is the Stokes operator, and
	\begin{align}
		C_{\ref{agmon}}(\s):= \max\left\{\frac{1}{\sqrt{-(2\s+1)}},\frac{1}{\sqrt{2\s+3}}\right\},
	\end{align}
\end{prop}

\begin{proof}
Let $u_k:=|\hat{u}(k)|$.  Now fix $\lam>0$ to be chosen later.  Observe that
	\begin{align}
		\sum_{k\in\Z^3}|k|^\s u_k=\underbrace{\sum_{|k|\leq\lam}|k|^\s u_k}_{A}+\underbrace{\sum_{|k|\geq\lam}|k|^{\s-1}|k|u_k}_{B}.\notag
	\end{align}
For $A$, we estimate as follows
	\begin{align}
		A\leq \left(\int_0^\lam r^{2\s+2} dr\right)^{1/2}\Sob{\ve{u}}{\ell^2}\leq\frac{1}{\sqrt{2\s+3}}\lam^{\s+3/2}\Sob{\ve{u}}{\ell^2}.\notag
	\end{align}
For $B$, we estimate
	\begin{align}
		B&\leq\left(\sum_{|k|>\lam}|k|^{2(\s-1)}\right)^{1/2}\left(\sum_{k\in\Z^3}|k|^2u_k^2\right)^{1/2}\notag\\
			&\leq c\left(\int_{\lam}^\infty r^{2\s}\ dr\right)^{1/2}\Sob{|\ \cdotp | \ve{u}}{\ell^2}\notag\\
			&\leq c\frac{1}{\sqrt{-(2\s+1)}}\lam^{\s+1/2}\Sob{|\ \cdotp| \ve{u}}{\ell^2}\notag,
	\end{align}
Combining $A$ and $B$ gives
	\begin{align}
		\sum_{k\in\Z^3}|k|^\s u_k\leq \max\left\{\frac{1}{\sqrt{-(2\s+1)}},\frac{1}{\sqrt{2\s+3}}\right\}\left(\lam^{\s+3/2}\Sob{\ve{u}}{\ell^2}+\lam^{\s+1/2}\Sob{|\ \cdotp|\ve{ u}}{\ell^2}\right).\notag
	\end{align}
Finally, choose 
	\begin{align}
		\lam:=\frac{\Sob{|\ \cdotp|\ve{u}}{\ell^2}}{\Sob{\ve{u}}{\ell^2}}.\notag
	\end{align}
Therefore
	\begin{align}
		\Sob{A^{\s/2}u}{\W}\leq C_\s\Sob{\ve{u}}{\ell^2}^{-(\s+1/2)}\Sob{|\ \cdotp|\ve{u}}{\ell^2}^{\s+3/2},\notag
	\end{align}
Rescaling with physical units and applying Parseval's identity completes the proof.
\end{proof}

\section{Acknowledgements}

This work of A. B. was supported in part by NSF grant number DMS-1109532, that of  V. M. and M. S. J. by DMS-1109638, and that of E.S.T. by DMS-1009950, DMS-1109640 and DMS-1109645, as well as the Minerva Stiftung/Foundation.

\bibliography{citations}

\begin{thebibliography}{100}

\bibitem{bfjr}
N.~Balci, C.~Foias, M.S.~Jolly, and R.~Rosa.
\newblock On universal relations in 2-{D} turbulence.
\newblock {\it Discrete and Continuous Dynamical Systems}, 2(4):1327--1351,
  August 2010.

\bibitem{bis2}
A.~Biswas.
\newblock Gevrey regularity for a class of dissipative equations with
  applications to decay.
\newblock {\it Journal of Differential Equations}, 253:2739--2764, 2012.

\bibitem{bs2}
A.~Biswas and D.~Swanson.
\newblock Existence and generalized {G}evrey regularity of solutions to the
  {K}uramoto-{S}ivashinsky equation in $\mathbb{R}^n$.
\newblock {\it Journal of Differential Equations}, 240(1):145--163, 2007.

\bibitem{bis}
A.~Biswas and D.~Swanson.
\newblock Gevrey regularity to the {3-D} {N}avier-{S}tokes equations with
  weighted $\ell^p$ initial data.
\newblock {\it Indiana University Mathematics Journal}, 56(3):1157--1188, 2007.

\bibitem{ccw}
P.~Constantin, D.~C\'ordoba, and J.~Wu.
\newblock On the critical dissipative quasi-geostrophic equation.
\newblock {\it Indiana University Mathematics Journal}, 50(1):97--107, 2001.

\bibitem{dfj1}
R.~Dascaliuc, C.~Foias, and M.S. Jolly.
\newblock Some specific {M}athematical {C}onstraints on {2-D} {T}urbulence.
\newblock {\it Physica D: Nonlinear Phenomena}, 237(23):3020--3029, 2008.

\bibitem{dfj2}
R.~Dascaliuc, C.~Foias, and M.S. Jolly.
\newblock On the asymptotic behavior of average energy and enstrophy in {3D}
  turbulent flows.
\newblock {\it Physica D}, 238:725--736, 2009.

\bibitem{doergibb}
C.R. Doering and J.D. Gibbon.
\newblock Applied analysis of the {N}avier-{S}tokes equations.
\newblock {\it Cambridge Texts in Applied Mathematics}, 1967.

\bibitem{dt}
C.~Doering and E.~Titi.
\newblock Exponential decay rate of the power spectrum for solutions of the
  {N}avier-{S}tokes equations.
\newblock {\it Phy. Fluids}, 7(6):1384--1390, June 1995.

\bibitem{hdong}
H.~Dong.
\newblock {D}issipative quasi-geostrophic equations in critical {S}obolev
  spaces: smoothing effect and global well-posedness.
\newblock {\it Discrete and Continuous Dynamical Systems}, 26(4):1197--1211,
  2010.

\bibitem{dongdong}
H.~Dong and D.~Li.
\newblock {S}patial analyticity of the solutions to the subcritical dissipative
  quasi-geostrophic equations.
\newblock {\it Arch. Rational Mech. Anal.}, 189:131--158, 2008.


\bibitem{fet}
A.~B. Ferrari and E.S. Titi.
\newblock {G}evrey regularity for nonlinear analytic parabolic equations.
\newblock {\it Communications in Partial Differential Equations},
  23(1\&2):1--16, 1998.

\bibitem{foias}
C.~Foias.
\newblock What do {N}avier-{S}tokes equations tell us about turbulence?
\newblock {\it Contemporary Mathematics}, 208:151--180, 1995.

\bibitem{foias:prodi}
C.~Foias and G.~Prodi.
\newblock Sur les solutions statistiques des \'equations de {N}avier-{S}tokes.
\newblock {\it Ann. Mat. Pura Appl.}, 11(4):307--330, 1976.

\bibitem{ft}
C.~Foias and R.~Temam.
\newblock Gevrey class regularity for the solutions of the {N}avier-{S}tokes
  equations.
\newblock {\it Journal of Functional Analysis}, 87:350--369, 1989.

\bibitem{fjmr}
C.~Foias, M.S.~Jolly, O.P.~Manley, and R.~Rosa.
\newblock Statistical estimates for the {N}avier-{S}tokes equations and the
  {K}raichnan theory of 2-{D} fully developed turbulence.
\newblock {\it Journal of Statistical Physics}, 108:591--645, 2002.

\bibitem{fmrt:book}
C.~Foias, O.P.~Manley, R.~Rosa, and R.~Temam.
\newblock Navier-{S}tokes equations and turbulence.
\newblock {\it Encyclopedia of Mathematics and its Applications}, 83, 2001.

\bibitem{fjmrt}
C.~Foias, M.S.~Jolly, O.P. Manley, R.~Rosa, and R.~Temam.
\newblock Kolmogorov theory via finite-time averages.
\newblock {\it Physica D}, 212(3-4):245--270, December 2005.

\bibitem{frisch}
U.~Frisch.
\newblock Turbulence: {T}he {L}egacy of {A}.{N}. {K}olmogorov.
\newblock {\it Cambridge University Press .}, 1995.

\bibitem{guotiti1}
P.~G\'erard, Y.~Guo, and E.S. Titi.
\newblock On the radius of analyticity of solutions to the cubic {S}zeg\"o equation.
\newblock {\underline{arXiv:1303.6148}}, pp.~1-12, ~2013.

\bibitem{gk}
Z.~Gruji\'c and I.~Kukavica.
\newblock Space analyticity for the {N}avier-{S}tokes and related equations
  with initial data in $l^p$.
\newblock {\it Journal of Functional Analysis}, 152:447--466, 1998.

\bibitem{guotiti2}
Y.~Guo and E.S. Titi.
\newblock Persistency of analyticity for quasi-linear wave-equations: an
  energy-like approach.
\newblock {\it Bulletin  of Institute of Mathematics, Academia Sinica (N.S.)}, pp. 1-27, 2013.

\bibitem{katznel}
Y.~Katznelson.
\newblock An introduction to harmonic analysis.
\newblock {\it Cambridge University Press}, 2004.

\bibitem{kiselev:note}
A.~Kiselev.
\newblock Some recent results on the critical surface quasi-geostrophic
  equation: a review.
\newblock {\it Proc. Sympos. Appl. Math.}, 67(Part 1):105--122, 2009.

\bibitem{kraich}
R.~Kraichnan.
\newblock Inertial ranges in two-dimensional turbulence.
\newblock {\it Phys. Fluids}, 10:1417--1423, 1967.

\bibitem{kuka}
I.~Kukavica.
\newblock On the dissipative scale for the {N}avier-{S}toke equation.
\newblock {\it Indiana University Mathematics Journal}, 47(3):1129--1154, 1998.

\bibitem{kv1}
I.~Kukavica and V.~Vicol.
\newblock On the radius of analyticity of solutions to the three-dimensional
  {E}uler equations.
\newblock {\it Proceedings of the American Mathematical Society},
  137(2):669--677, 2009.

\bibitem{lt}
A.~Larios and E.S. Titi.
\newblock On the {H}igher-order global regularity of the inviscid
  {V}oigt-regularization of three-dimensional hydrodyanmic models.
\newblock {\it Discrete and Continuous Dynamical Systems Series B},
  14(2):603--627, 2010.

\bibitem{lo}
C.~D. Levermore and M.~Oliver.
\newblock Analyticity of solutions for a generalized {E}uler equation.
\newblock {\it Journal of Differential Equations}, 133:329--339, 1997.

\bibitem{marchioro}
C.~Marchioro.
\newblock An example of turbulence for any {R}eynolds number.
\newblock {\it Comm. Math. Phys.}, 108(4):647--651,1987

\bibitem{vthesis}
V.~Martinez.
\newblock {\it {G}evrey {R}egularity of {N}avier-{S}tokes and {Q}uasi-geostrophic equations with {A}pplications to 2{D} and 3{D} turbulence}
\newblock {Department of Mathematics, Indiana University}, 2014, in preparation.

\bibitem{ot2}
M.~Oliver and E.~S. Titi.
\newblock {R}emark on the rate of decay of higher order derivatives for
  solutions to the {N}avier-{S}tokes equations in $\mathbb{R}^n$.
\newblock {\it Journal of Functional Analysis}, 172:1--18, 2000.

\bibitem{ot}
M.~Oliver and E~.S. Titi.
\newblock {O}n the domain of analyticity for solutions of second order analytic
  nonlinear differential equations.
\newblock {\it Journal of Differential Equations}, 174:55--74, 2001.

\bibitem{silv1}
L.~Silvestre.
\newblock On the differentiability of the solution to an equation with drift
  and fractional diffusion.
\newblock {\it Indiana University Mathematical Journal.}, 61(2):557--584, 2012.

\bibitem{silv:vicol:zlat}
L.~Silvestre, V.~Vicol, and A.~Zlatos.
\newblock On the loss of continuity for super-critical drift-diffusion
  equations.
\newblock {\underline{ arXiv:1205.4364v2}}, pp.1-27, 2012.

\bibitem{temamnse}
R.~Temam.
\newblock Navier-{S}tokes equations: theory and numerical analysis.
\newblock {\it North-Holland Publishing Company}, 1977.

\bibitem{weissler1}
F.~B. Weissler.
\newblock The {N}avier-{S}tokes initial value problem in $l^p$.
\newblock {\it Archive for Rational Mechanics and Analysis}, 74:219--230, 1980.

\end{thebibliography}

\end{document}